\documentclass{amsart}
\usepackage{amsmath} 
\usepackage{amssymb}
\usepackage{mathtools}
\usepackage{centernot}
\usepackage{stmaryrd}
\usepackage{esvect}
\usepackage{amsthm}
\usepackage{upgreek}
\usepackage{comment}
\usepackage{amsfonts}
\usepackage{physics}
\usepackage{tikz-cd}
\usepackage[utf8]{inputenc}
\usepackage[english]{babel}
\usepackage{graphicx}
\usepackage[margin=1in]{geometry}
\usepackage{mathtools}
\usepackage{hyperref}
\usepackage{microtype}

\graphicspath{{images/}}
\newcommand{\R}{\mathbb{R}}

\renewcommand{\part}[1]{\noindent\textbf{Part #1)}}

\newcommand{\ba}{\begin{align*}}

\newcommand{\ea}{\end{align*}}

\newcommand{\N}{\mathbb{N}}

\newcommand{\Lam}{\Lambda}

\newcommand{\Int}[1]{%
  {\kern0pt#1}^{\mathrm{o}}%
}

\newcommand{\F}{\mathcal{F}}

\newcommand{\bp}{\begin{pmatrix}}
\newcommand{\ep}{\end{pmatrix}}

\newcommand{\interior}[1]{%
  {\kern0pt#1}^{\mathrm{o}}%
}

\newcommand{\cLamu}{\widetilde{\Lam}^u}
\newcommand{\cLams}{\widetilde{\Lam}^s}
\newcommand{\cSigs}{\widetilde{\Sigma}^s}
\newcommand{\cSigu}{\widetilde{\Sigma}^u}
\newcommand{\Lamu}{\Lambda^u}
\newcommand{\Lams}{\Lambda^s}
\newcommand{\sat}{\text{Sat}}

\usepackage{xargs}                      
\usepackage{xcolor}
\usepackage[colorinlistoftodos,prependcaption,textsize=tiny]{todonotes}
\newtheorem{theorem}{Theorem}[section]
\newtheorem{corollary}[theorem]{Corollary}

\newtheorem{lemma}[theorem]{Lemma}
\newtheorem{prop}[theorem]{Proposition}

\theoremstyle{definition}
\newtheorem{definition}[theorem]{Definition}
\newtheorem{remark}[theorem]{Remark}

\title{Finiteness of Anosov and Pseudo Anosov flows on hyperbolic 3-manifolds}

\author[D. Gabai]{David Gabai}
\address{Department of Mathematics, Princeton University, Fine Hall, 304 Washington Rd, Princeton, NJ 08544}
\email{gabai@math.princeton.edu}

\author[L. Li]{Lily Qiao Li}
\address{Department of Mathematics, Princeton University, Fine Hall, 304 Washington Rd, Princeton, NJ 08544}
\email{lilyli@princeton.edu}

\subjclass[2020]{}

\begin{document}  

\begin{abstract} We show that for any closed hyperbolic 3-manifold $M$, there are only finitely many Anosov or pseudo-Anosov flows on $M$ up to orbit equivalence.
\end{abstract}

\maketitle

\section{Introduction}
In this paper, we show the finiteness conjecture for Anosov and pseudo-Anosov flows on closed hyperbolic 3-manifolds. In particular, we prove the following result.

\begin{theorem}\label{Theorem: main theorem}
    Let $M$ be a closed hyperbolic manifold. Then there are finitely many Anosov and pseudo-Anosov flows up to orbit equivalence on $M$. 
\end{theorem}



In dimension 3, Anosov flows are flows on 3-manifolds that are locally modeled on suspension flows of Anosov homeomorphisms on the torus, with local expansion-contraction properties. Pseudo-Anosov flows behave like Anosov flows away from finitely many singular orbits. They were originally introduced by Thurston as a natural generalization of Anosov flows. Informally, they are locally modeled on suspension flows of pseudo-Anosov homeomorphisms on a surface of genus $g$.

Smale proposed a program to study Anosov flows on closed manifolds as part of his 1966 ICM address. His program concerns existence, abundance, and classification. Since then, the study of Anosov and pseudo-Anosov flows have proven challenging yet fruitful in dimension 3.

Pseudo-Anosov flows have been shown to be abundant on various classes of 3-manifolds. Work of B\'eguin--Bonatti--Yu \cite{BBYabundance} establishes that for any $n\in\N$, there exists 3-manifolds admitting $n$ different pseudo-Anosov flows up to orbit equivalence. Bowden--Mann \cite{BMabundance}, Clay--Pinsky \cite{CPabundance}, B\'eguin--Yu \cite{BYabundance}, and Baldwin--Sivek--Zung \cite{BSZabundance} establish analogous abundance results for various families of 3-manifolds.

 In the 70s and 80s, work of Plante \cite{Plante}, Ghys \cite{Ghys1984}, Barbot \cite{Barbot1996}, and Fried--Thurston \cite{Fried1979} showed finiteness for pA flows in solvable manifolds and Seifert-fibered manifolds. In recent years, there has been much activity and progress on the finiteness conjecture, drawing from techniques at the intersection of several fields of study spanning low-dimensional topology, dynamics, and contact geometry. Work of Barthelm\'e--Bowden--Mann proved the finiteness conjecture for Reeb Anosov flows using the classification of transitive pseudo-Anosov flows up to orbit equivalence via free homotopy classes of closed orbits due to Barthelm\'e--Mann and Barthelm\'e--Frankel--Mann \cite{BM24}\cite{BFMclassification}. Zung \cite{Zungfiniteness} and Baldwin--Sivek--Zung \cite{BSZabundance} extended this to show finiteness for pseudo-Anosov flows admitting positive Birkhoff sections in rational homology spheres. Chaidez--Pan \cite{ChaidezPan} proved the finiteness conjecture for pseudo-Anosov Reeb flows, which they define. Barthelm\'e--Tsang--Zung \cite{noperfectfits} proved finiteness for pseudo-Anosov flows without perfect fits. Barthelm\'e--Paulet \cite{BPfiniteness} show a virtual finiteness result for transitive pseudo-Anosov flows on graph manifolds.

\subsection*{Comments about the general case in dimension 3} Though this paper is focused on closed hyperbolic 3-manifolds, the authors expect that with elaborations of branched surface theory implicit in the work of Tao Li on essential branched surfaces in general manifolds \cite{TaoLiHeegaard} and/or suitable generalizations of Kneser normal form \cite{kneser}, the main result can be obtained for toroidal manifolds up to orbit equivalence. We also expect the same result for finite volume cusped manifolds.

\subsection*{Relationship to other works} The second author presented a proof of this result at a talk given at the Oberwolfach workshop on low-dimensional topology on July 7, 2026. Jonathan Bowden and Vincent Colin posted a preprint \cite{bowdenColin} using different methods to prove similar results.

\subsection*{Idea of the Proof}

The proof of Theorem \ref{Theorem: main theorem} occupies Sections \ref{section: anosov to pA} to \ref{section: local lemma}.

First, to treat the Anosov flows case, we show in Section \ref{section: anosov to pA} that an infinite sequence of Anosov flows on $M^3$ has a subsequence sharing a common closed orbit up to isotopy. Denjoy splitting near this shared orbit produces stable and unstable essential laminations. This step allows us to analyze both Anosov and pseudo-Anosov flows from a pair of transverse laminations.

Then, we put both laminations into Kneser normal form simultaneously in Section \ref{section: sim norm} while preserving their transverse intersections that are flow lines. In the same section, the simultaneous normalization also gives us our first source of finiteness: there are only finitely many pairs of essential normal branched surfaces, Kneser branched surfaces, that carry the stable and unstable laminations. However, this does not prove the finiteness of flows because a fixed pair of essential branched surfaces could carry many different laminations within their fibered neighborhoods. The rest of the proof seeks to reconcile this issue.

To this end, we consider that one key difference between the stable/unstable laminations and the Kneser branched surfaces is the following: the pair of laminations is transverse, but the branched surfaces carrying them may not be transverse. One concern is that perhaps the Kneser branched surfaces could carry stable and unstable laminations with leaves that travel through the same sequences of normal disks, such that the leaves look parallel. In Section \ref{section: parallel locus}, we make this precise by defining the parallel locus and argue that parallel disks of increasing radius in the stable and unstable laminations converge to a common stable/unstable limit sublamination. 

We then argue in Section \ref{section: swaths} that such a common sublamination cannot exist by considering two cases. Either the common sublamination has holonomy or does not. If it has no holonomy, we use an Euler class argument in Subsection \ref{subsection: no holonomy euler class} to reach a contradiction. If it has holonomy, we produce incompatible stable and unstable perfect fits behavior in Subsection \ref{subsection: holonomy swaths} also to reach a contradiction. Therefore, the stable and unstable laminations cannot share large parallel disks after all.

The only remaining possible parallel behavior would be if stable and unstable leaves share long thin corridors. In Section \ref{section: corridor to local}, we construct a finite arc system in Proposition \ref{prop: uniformly finite clean enpoints} and \ref{prop: unifromly finite mother graph of all cuts} to cut each remaining corridor into disks using a uniformly bounded number of cuts each of uniformly bounded length. After splittings and isotopies described in Subsection \ref{subsection: isotoping near cutting arcs}, the stable Kneser branched surface decomposes into finitely many disks, each satisfying the property that along the boundaries of the disks, the unstable Kneser branched surface intersects transversely finitely many times. In Section \ref{section: local lemma}, we prove the Local Lemma in Theorem \ref{thm: local lemma}, which describes how the transverse boundary behavior implies that in a disk region of the stable branched surface, the stable and unstable laminations interact as a trivial product. We then conclude there are only finitely many such branched surface pairs carrying the laminations. These are the first transverse Kneser branched surface pairs we obtain.

Consequently, we conclude that there are only finitely many transverse pairs of stable and unstable Kneser branched surfaces. Finally, we argue that such a transverse pair of branched surfaces determines the flow. In particular, we argue in Proposition \ref{prop: transverse Kneser pair determines flow} that two flows carried by the same transverse pair of branched surfaces have the same spectrum of free homotopy classes of closed orbits. Therefore, using work of Barthelm\'e--Frankel--Mann \cite{BFMclassification}, we have that finitely many transverse Kneser branched surface pairs implies finitely many Anosov and pseudo-Anosov flows.

We use the hyperbolicity of the manifold in a few places. For one, the simultaneous normalization result in Section \ref{section: sim norm} relies on Theorem \ref{thm: Gabai kneser normal form} of the first author, which requires the manifold to be atoroidal. We also use that essential branched surfaces in hyperbolic manifolds do not carry tori in getting a contradiction to the no holonomy limit sublamination in Subsection \ref{subsection: no holonomy euler class}. Lastly, we use that pseudo-Anosov flows on hyperbolic manifolds do not admit trees of scalloped regions when applying Barthelm\'e--Frankel--Mann in Proposition \ref{prop: transverse Kneser pair determines flow}.


\subsection*{Acknowledgments} The authors thank Ian Agol, Sergio Fenley, Katie Mann, Rafael Potrie, Chi Cheuk Tsang, and Mehdi Yazdi for helpful conversations and comments. Special thanks to Chi Cheuk Tsang for numerous constructive comments on an initial draft.

DG was partially supported by the National Science Foundation Grants DMS-2304841 and DMS-2603581. LL was partially supported by the National Science Foundation Grant No. DMS-2304841. Both authors were partially supported by Grant No. DMS-1928930 while in residence at the Simons Laufer Mathematical Sciences Institute in Berkeley, California, during the semester of Spring 2026, where DG was a Clay Senior Scholar and also partially supported by the Clay Foundation and LL was a Program Associate.  The authors are grateful for the stimulating environment and hospitality at SLMath. LL was also supported by a Simons Dissertation Fellowship in Mathematics SFI-MPS-SDF-00015436.

\section{Preliminaries}\label{section: prelim}
\subsection{Anosov and pseudo-Anosov flows}
We give a definition of a topological pseudo-Anosov flow following \cite{BMbook}.
\begin{definition}[Pseudo Anosov Flows]
    A \textit{pseudo-Anosov flow} is a flow $\varphi:M\times \R\rightarrow M$ satisfying the following properties:
    \begin{enumerate}
        \item There are two transverse 2D singular foliations $\mathcal{F}^s, \mathcal{F}^u$ whose leaves are saturated by orbits of $\varphi$ and intersect along orbits. They are called the \textit{stable and unstable foliations}, respectively.
        \item Singularities of $\mathcal{F}^s$ and $\mathcal{F}^u$ coincide and lie along a finite collection of periodic orbits $\alpha_i$. We call these orbits \textit{singular orbits}. A neighborhood of each singular orbit is homeomorphic to a mapping torus of a pseudo-Anosov homeomorphism of a surface near a fixed $p$-pronged singularity. See Figure \ref{fig:prongs denjoy} for an illustration of $\mathcal{F}^s$ and $\mathcal{F}^u$ near a singular orbit.
        \item In a leaf of the stable foliation $\mathcal{F}^s$, any two flow lines converge in forward time and diverge in backward time.
        \item In a leaf of the unstable foliation $\mathcal{F}^u$, any two flow lines diverge in forward time and converge in backward time.
    \end{enumerate}
\end{definition}

\begin{figure}
    \centering
    \includegraphics[width=0.13\linewidth]{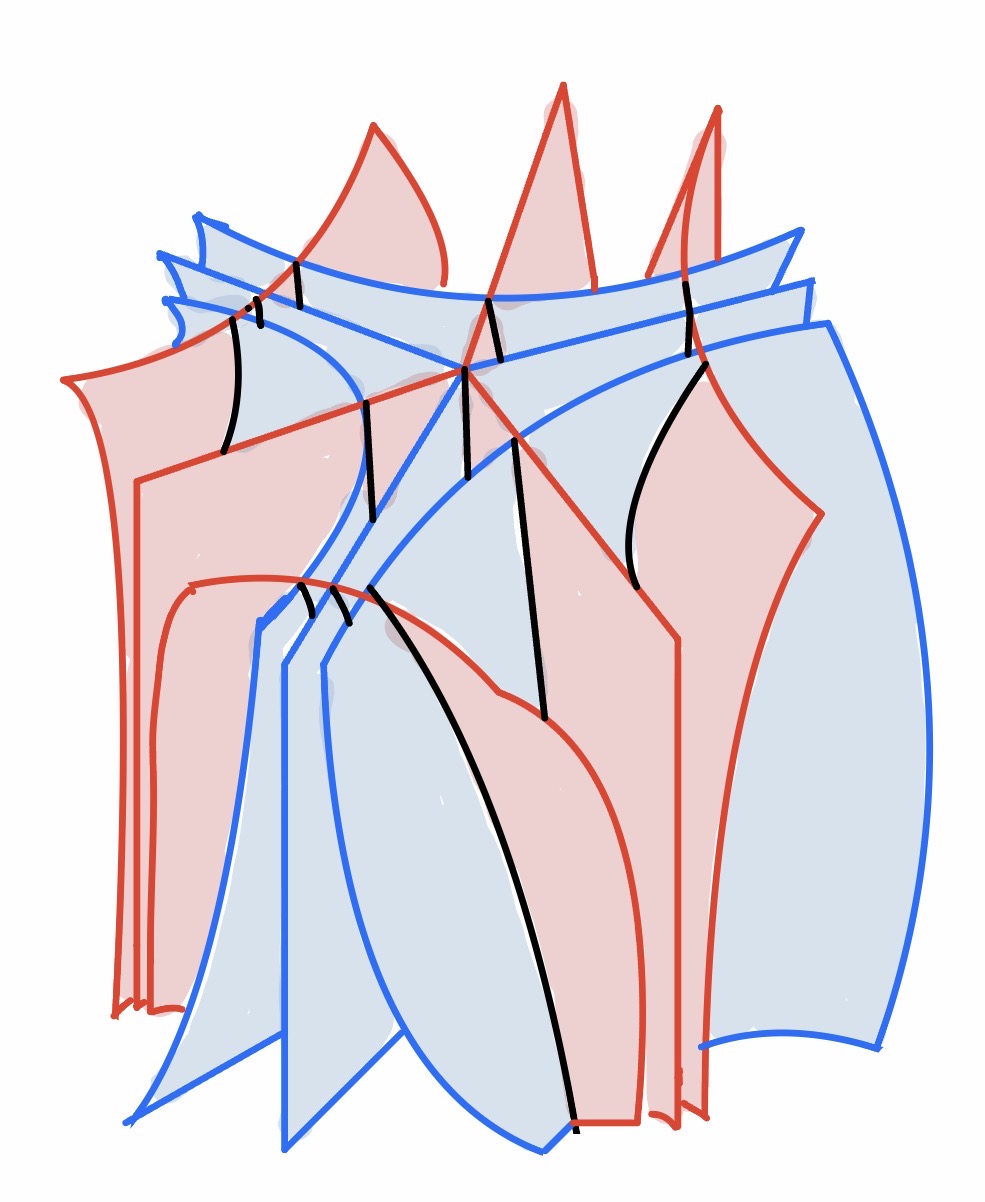}
    \hspace{3em}
    \includegraphics[width = 0.1\linewidth]{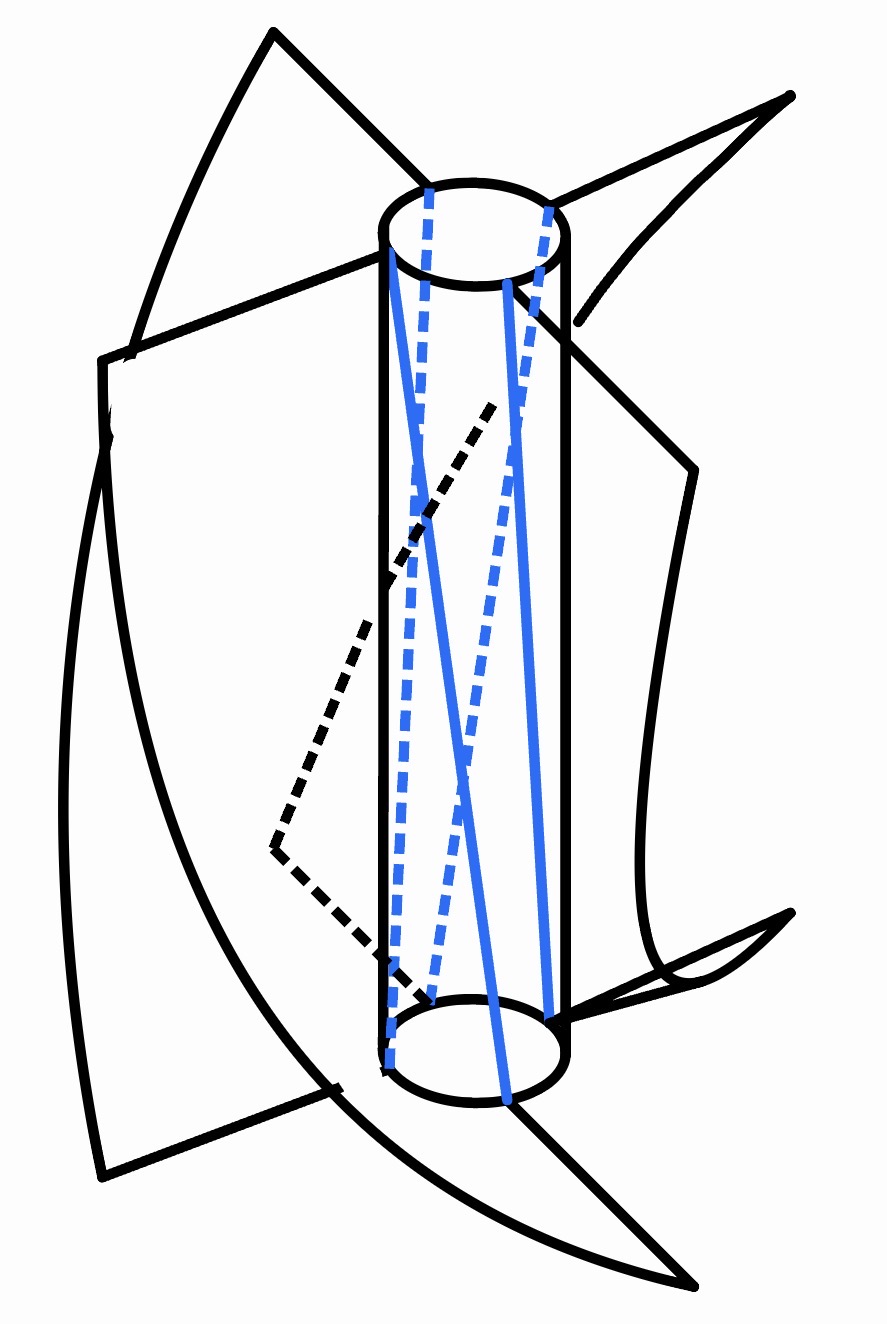}
    \caption{Stable and unstable foliations of a pseudo-Anosov flow near a singular orbit.}
    \label{fig:prongs denjoy}
\end{figure}

\subsection{Normal Surfaces}

Next, we record two theorems of the first author regarding our first source of finiteness. 

\begin{theorem}[Gabai \cite{kneser}, Corollary 6.21]\label{thm: Gabai kneser normal form}
    Let $M$ be a closed orientable atoroidal 3-manifold. There exists a triangulation $\tau$ on $M$ such that any taut foliation or Reebless foliation or nowhere dense essential lamination can be isotoped to be normal to $\tau$.
\end{theorem}

\begin{theorem}[Gabai \cite{kneser}, Theorem 6.5]\label{thm: Gabai finite essential branched surface no tori carried}
    Let $M$ be a closed orientable atoroidal 3-manifold. The collection of nowhere dense essential laminations on $M$ is carried, up to isotopy, by finitely many essential branched surfaces.
\end{theorem}
\begin{remark}
    In closed hyperbolic 3-manifolds, essential branched surfaces do not carry tori.
\end{remark}

\subsection{Limits of Laminations}
Next, we record two results in lamination theory about when a sequence of laminations converge to a limit lamination.

\begin{lemma}\label{lemma: laminations converge to limit}
    Given a sequence of essential laminations $\Lam_i$ carried by the same essential branched surface $\Sigma$, there is a subsequence $\{\Lam_j\}$ that converges to a limit lamination $\Lam_{\lim}$, possibly with branching. Further, we can arrange for the tangent bundles $T\Lam_j$ to also converge to $T\Lam_{\lim}$.\qed
\end{lemma}

\begin{remark}
    The limit $\Lam_{\lim}$ may have branched leaves that require splitting to obtain an essential lamination carried by $\Sigma$. See, for example, work of the first author\cite{kneser}.
\end{remark}

\begin{lemma}\label{lemma: disk converge to lam}
    Let $\Sigma$ be a fixed branched surface and let $D_i$ be radius-$R_i$ metric disks contained in laminations $\Lam_i$ carried by $\Sigma$. If $R_i\rightarrow \infty$, then up to passing to a subsequence and possibly taking sub-disks, the disks $D_i$ converge uniformly on bounded leafwise balls to a complete leaf $L_\infty$ carried by $\Sigma$. The closure of $L_\infty$ is a lamination carried by $\Sigma$. Further, if the laminations $\Lambda_i$ and the branched surface $\Sigma$ are normal to a triangulation $\tau$, then the convergence can be made $C^1$ on compact sets.\qed
\end{lemma}

\subsection{Finiteness of Singular orbits}
We also record a result of the second author to use in various parts of the proof.

\begin{theorem}[Li \cite{LiSingularOrbit}, Theorem B]\label{thm: finite singular orbits}
    Let $M$ be a closed, atoroidal 3-manifold. There is a triangulation $\tau$ so that there are only finitely many triangulated solid tori in $\tau$ that can arise as gut regions of stable laminations of pseudo-Anosov flows on $M$.
\end{theorem}

\section{Reducing the Anosov case to pseudo-Anosov}\label{section: anosov to pA}

In this section, we show that if there were infinitely many Anosov flows on a hyperbolic 3-manifold $M$ up to orbit equivalence, then infinitely many share a common closed orbit. Up to taking a double cover, we assume that all the laminations considered in this section are orientable. 

\begin{theorem} \label{theorem: reduce anosov to pA}
    Let $\{\phi_i\}$ be a sequence of Anosov flows on a hyperbolic 3-manifold $M$. Then there is a subsequence $\{\phi_{i_k}\}$ that all share a common closed orbit up to isotopy. 
\end{theorem}

First, we establish some results about laminations.

\begin{lemma}\label{lemma: homotopic plane fields}
    If two oriented 2-plane fields $\xi_1, \xi_2$ on a closed 3-manifold satisfy 
    \[\angle(\xi_1(x), \xi_2(x))<\pi/2\]
    for every $x\in M$, then $\xi_1, \xi_2$ are homotopic through oriented plane fields.\qed
\end{lemma}

    Let $\{\phi_i\}$ be a sequence of Anosov flows on a hyperbolic 3-manifold $M$. Let $\F^s_i$ be the stable foliation of $\phi_i$. For each $i$, choose a closed orbit and Denjoy split to obtain an essential lamination $\Lam^s_i$. By Gabai's Kneser normal form theorem \cite{kneser}, there are only finitely many essential branched surfaces carrying all the essential laminations in $M$. By the pigeonhole principle, there is a subsequence carried by the same essential branched surface $\Sigma^s$. Then, Lemma \ref{lemma: laminations converge to limit} implies there is a further subsequence that converges to a limit lamination $\Lam_{\lim}$. 

    Replace $\Lam_{\lim}$ by a minimal sublamination of itself, as needed.

\subsection*{Case I: $\Lam_{\lim}$ has no holonomy}\label{subsection: anosov no holonomy}
    If $\Lam_{\lim}$ has no holonomy, there exists a transverse invariant measure on $\Lam_{\lim}$ by work of Candel \cite{Candel1999} generalizing a theorem of Plante \cite{Plante1975}. Such an invariant measure corresponds to a real-valued solution $\vec{x}$ to the branched surface equations given by $\Sigma^s$. Suppose that the system of linear equations give the matrix $A$, then the invariant measure gives $A\vec{x}=0$. 

    Consider the 2-plane field $T\Lam_{\lim}$ to which $\Lam_{\lim}$ is tangent. We show that we can construct two approximating 2-plane fields homotopic to $T\Lam_{\lim}$ with different Euler classes, which gives the contradiction.

        Consider the sequence $\Lam^s_i$ converging to $\Lam_{\lim}$, and extend the tangent 2-plane field $T\Lam^s_i$ to $U$ for each $i$. 
    
    \begin{lemma}\label{lemma: extend lim lam tangent to neighborhood}
        For a sufficiently small neighborhood $U$, there is an extension $\xi_i$ of the 2-plane field $T\Lam^s_i$ to $U$ so that the Euler class $e(\xi_i) = 0$.
    \end{lemma}
    \begin{proof}
        Consider the weak stable foliation of the sequence of Anosov flows $T\mathcal{F}^{ws}_i=E^s_i\oplus\langle X_i\rangle$. After Denjoy splitting, we consider each complementary region of $\Lam_i$ as a copy of the product $P=L\times [0,1]$ for some leaf $L\subset\mathcal{F}^{ws}_i$. Then we decompose the tangent space at $P$ as $T_{(x,t)}P = T_xL\times T_t[0,1]$ and define a plane field on $P$ by $(\xi_i)_{(x,t)}=T_xL\times \{0\}$. It's compatible with the tangent 2-plane field $T\Lam^s_i$, namely that $T\Lam^s_i$ matches $\xi$ on the boundary leaves $L\times \{0\}, L\times \{1\}$. Therefore $\xi$ and $\Lam^s_i$ glue to a 2-plane field on a neighborhood $U$ of $\Lam_{\lim}$.

        The Anosov vector field in the flow direction also extends into the complementary region. Define $\widetilde{X}$ on the product region $P$ by $\widetilde{X}(x,t)=(X(x), 0)$. On the points of $\Lam^s_i$, assign $\widetilde{X}|_{\Lam^s_i}=X|_{\Lam^s_i}$. Then $\widetilde{X}$ defines a nowhere vanishing vector field tangent to $\xi_i$, which is a global section $\widetilde{X}\in\Gamma(\xi_i)$. Since $\xi_i$ is a 2-plane bundle with a nonzero section, we conclude that $e(\xi_i)=0$.
    \end{proof}

    Next, we extend $T\Lam_{\lim}$ to $U$, call it $\xi_{\lim}$, such that for sufficiently large $i$, $\xi_i, \xi_{\lim}$ are homotopic. To do this, consider that from the convergence of the stable plane fields $\xi_i$ to $\Lam_{\lim}$, for sufficiently large $i$, we have
    \[\angle(T_x\Lam_{\lim}, (\xi_i)_x)<\frac{\pi}{2}.\]
    Then since $\Lam_{\lim}$ is closed in $U$ and the condition $\angle(T_x\Lam_{\lim}, (\xi_i)_x)<\frac{\pi}{2}$ is open, we can extend the 2-plane field $T\Lam_{\lim}$ to $\xi_{\lim}$ on $U$ after shrinking $U$ as necessary so that for every $x\in U$, we have
    \[\angle(\xi_i, \xi_{\lim})<\frac{\pi}{2}.\]

    By Lemma \ref{lemma: homotopic plane fields}, $\xi_i$ and $\xi_{\lim}$ are homotopic as 2-plane fields in $U$. Since $e(\xi_i)=0$, we conclude that $e(\xi_{\lim}) = 0$. We have thus constructed the first plane field $\xi_{\lim}$ on a small neighborhood $U$ of $\Lam_{\lim}$.

    The second approximating 2-plane field $\xi_\epsilon$ is constructed from modifying $\xi_{\lim}$ using the sequence of rational approximations $\vec{x_i}$. 
        
        The following lemma establishes that integral solutions to the branch equations of a splitting of $\Sigma$ give rise to carried surfaces with tangent field closely approximating $T\Lam_{\lim}$.
\begin{lemma}\label{lemma: construct B epsilon}
    Let $\Sigma$ be a branched surface carrying a lamination $\Lambda$ and $\xi$ be a plane field extending $T\Lam$ to a fibered neighborhood $N(\Sigma)$ of $\Sigma$, so that $\xi=T\Lam$ where defined. For any $\epsilon>0$, we can construct a splitting of $\Sigma$ into a smooth branched surface $\Sigma_\epsilon\subset U$ carrying $\Lam_{\lim}$ with the following properties:
    \begin{enumerate}
        \item For any point $x$ in $\Sigma_\epsilon$, we have 
        \[\angle(T_x\Sigma_\epsilon, \xi_x)<\epsilon;\]
        \item For any nonnegative integral weight $w$ on $\Sigma_\epsilon$, there is a smooth carried surface $S_w\subset N(\Sigma_\epsilon)$ realizing $w$ that satisfies
        \[\angle(T_xS_w, \xi_x)<\epsilon.\]
    \end{enumerate}\qed
\end{lemma}

    Fixing $\epsilon<\frac{\pi}{2}$, construct the corresponding branched surface $\Sigma_\epsilon$ using Lemma \ref{lemma: construct B epsilon}. Since $\Sigma_\epsilon$ still carries the lamination $\Lam_{\lim}$, we obtain a positive real-valued solution $\vec{x_\epsilon}\in\R^k$ to the branch equations $A_\epsilon\vec{x_i}=0$ subject to the branching behavior of $\Sigma_\epsilon$. There is a sequence of rational approximations $\vec{x_i}$ so that $A_\epsilon\vec{x_i}=0$ for all $i$, and $\lim_{i\rightarrow\infty} \vec{x_i}=\vec{x_\epsilon}$. Each  $\vec{x_i}$ gives a lamination $\Lam_i$ by closed embedded surfaces. Given a rational approximation $\vec{x_i}$, we can clear denominators to obtain integer weights $w_i$. Then invoke Lemma \ref{lemma: construct B epsilon} to construct smooth surfaces $S_{w_i}$ realizing $w_i$ and satisfying 
    \[\angle(T_xS_{w_i}, (\xi_{\lim})_x)<\epsilon.\]

    Since $\epsilon<\frac{\pi}{2}$, Lemma \ref{lemma: homotopic plane fields} shows that $TS_{w_i}, (\xi_{\lim})|_{S_{w_i}}$ are homotopic as 2-plane fields on $S_{w_i}$. Now homotope $\xi_{\lim}$ in a neighborhood of $S_{w_i}\subset U$ to match $TS_{w_i}$, and we obtain a homotopic plane field $\xi_\epsilon$ on $U$. 

    Consider that the bundle $TS_{w_i}$ is the pullback bundle $i^*\xi_{\lim}$ where $i:S_{w_i}\hookrightarrow U$ is the inclusion map. By Gabai's result as recorded in Theorem \ref{thm: Gabai finite essential branched surface no tori carried}, the branched surface $\Sigma$ carries no spheres or tori, which implies that its splitting $\Sigma_\epsilon$ also carries no spheres or tori. Therefore, any surface $S_{w_i}$ must be of genus at least 2, with Euler characteristic $\chi(S_{w_i})\neq 0$. Therefore, the Euler class $e(TS_{w_i})\neq 0$. This is a contradiction because as a pullback bundle, $TS_{w_i}$ should satisfy $e(TS_{w_i}) = e(i^*\xi_{\lim})=0$. This is a contradiction.

\subsection*{Case II: $\Lam_{\lim}$ has nontrivial holonomy} 

In this subsection, we establish the following result. 

\begin{prop}
    Let $\Lam^s_i$ be the subsequence of stable laminations converging to the limit lamination $\Lam^s_{\lim}$. If there is an oriented closed embedded curve $\gamma_0\subset L$ in $\Lam_{\lim}$ with non-trivial holonomy, then there exists $N\in\N$ such that for all $i>N$, each $\Lam^s_i$ contains a closed curve $\gamma_i\subset L_i$ representing the same isotopy class $[\gamma_i]=[\gamma_j]$.
\end{prop}

By definition, we can find a transversal $I_0$ to $\gamma_0$ so that the holonomy transport gives a map $f_{\gamma_0}:I_0\rightarrow I_0$ satisfying $f_{\gamma_0}(I_0)\subset I_0$. We can further find a fixed point of $f_{\gamma_0}$ and obtain a closed curve with one-sided contracting holonomy. Replace $\gamma_0$ with that curve.

Let $I_1$ be the component of $I_0\setminus f_{\gamma_0}(I_0)$ lying to the contracting side of $\gamma_0$. Since the limit lamination is minimal, the leaf $L$ containing $\gamma_0$ meets the interior of the transversal $I_1$. There is a leafwise path $\gamma_1\subset L$ from $p_0\subset \gamma_0\cap I_0$ to $pi\in I_1\cap L$. Consider a smaller transversal $I_2\subset I_0$ so that the parallel transport $f_{\gamma_1}(I_2)\subset I_1$. See Figure \ref{fig: holonomy spinning} for an illustration. 

In the sequence $\{\Lam_i^s\}$, there exists $N_1\in\N$ so that for $i>N_1$, we can perform an ambient isotopy so that $\gamma_1\subset \Lam^s_i$. Choose a compact tubular neighborhood $U$ of $\gamma_1$ and arrange the isotopies to be $C^1$ small and supported on $U$. Then, after possibly shrinking $I_2$ so that the closure of $f_{\gamma_1}(I_2)$ is contained in the interior of $I_1$, the holonomy maps $f_{\gamma_1}^{s,i}:I_2^{(i)}\rightarrow I_1$ satisfy $f_{\gamma_1}^{s,i}(I_2^{(i)})\subset I_1$. 

Since $\gamma_0$ in the limit has one-sided contracting holonomy, and $I_1, I_2$ lie on the contracting side, there exists $k\in \N$ so that an iterate $f_{\gamma_0}^k(I_1)\subset I_2$. This means $F = f_{\gamma_0}^k\circ f_{\gamma_1}:I_2\rightarrow I_2$ satisfies $F(I_2)\subset I_2$. For sufficiently large i, after the small ambient isotopies identifying the same curves $\gamma_0, \gamma_1$ with curves lying in $\Lam_i^s$, the corresponding holonomy maps satisfy $(f_{\gamma_0}^{s,i})^k(I_1)\subset I_2$ for $i>N_2$. For $i>\max(N_1, N_2)$, we have the holonomy transport $F^s_i=(f_{\gamma_0}^{s,i})^k\circ f_{\gamma_1}^{s,i}:I_2\rightarrow I_2$ satisfying $F_i^s(I_2)\subset I_2$. Since $I_2$ is compact, by the Brouwer fixed point theorem, there is a point $p_i^s\in I_2$ with $F_i^s(p_i^s)=p_i^s$, and the parallel transport along $p_i^s$ yields a closed curve $\gamma^s_i\subset\Lam^s_i$ for each $i$. This curve first follows $\gamma_1$ and then winds around a local copy of $\gamma_0$ $k$ times, so all $\gamma_i^s$ are of the same isotopy class. We have found curves $\gamma_i^s\subset \Lam^s_i$ such that $[\gamma_i^s]=[\gamma_j^s]$ for all sufficiently large $i,j>\max(N_1, N_2)$.

For sufficiently large $i$, the curve $[\gamma_i^s]$ is a transverse pushoff of the curve $[\gamma^s]$ obtained in the limit from a fixed point of $F: I_2\rightarrow I_2$. So $[\gamma_i^s]$ are further isotopic.

The leaves $L_i$ containing the curves $[\gamma_i^s]$ are therefore annular leaves with $[\gamma_i^s]$ as their core curves. By the Anosov dynamics, the core curve of an annular stable leaf represents a closed orbit of the flow. Therefore, for sufficiently large $i$, the flows $\phi_i$ share an isotopic closed orbit.

In the pseudo-Anosov case, we use a result of the second author \cite{LiSingularOrbit} to establish that an infinite sequence of pseudo-Anosov flows on the same hyperbolic manifold will also contain a subsequence with isotopic links arising as singular orbits.

\begin{theorem}[Li \cite{LiSingularOrbit}, Theorem A]
    Let $M$ be a closed atoroidal manifold. Then the class of isotopy classes of curves representing singular orbits of pseudo-Anosov flows on M is finite. 
\end{theorem}

In either case, now we have canonical choices for which orbit(s) to perform Denjoy splitting along to obtain a pair of stable and unstable essential laminations $\Lam^s, \Lam^u$.

\section{Simultaneous Normalization}\label{section: sim norm}

In this section, we show that we can simultaneously isotope $\Lam^u, \Lam^s$ to be normal to a triangulation that depends only on the ambient manifold $M$.

\begin{theorem}\label{simnorm}
    Let $M$ be a closed orientable atoroidal 3-manifold. There exists a triangulation $\tau$ such that any pair $\Lam^s, \Lam^u$ of transverse essential laminations coming from stable and unstable laminations of pseudo-Anosov flows can be simultaneously isotoped to be normal to $\tau$ while respecting their transversality.
\end{theorem}

This is a pairs version of Theorem \ref{thm: Gabai kneser normal form}. 

\noindent \textit{Proof.} The proof of Theorem \ref{simnorm} consists of four steps. Each will make use of the following bigon reduction procedure. We will describe bigon reduction in Step 0, but the proof begins at Step 1, with references to do Step 0 as needed. 

\subsection*{Step 0: Bigon reduction}\label{procedure: bigon reduction}
    We describe how to eliminate a bigon with edges $e^s, e^u$ on the stable and unstable laminations, respectively. Consider a bigon $B$ to have vertices $v, v'$ and edges $e^s, e^u$. Usually, bigons come in product families, but we pick the innermost bigon so that its interior disjoint from $\Lam^s, \Lam^u$. In practice, we can perform bigon reduction on a parallel family analogously by simultanenously eliminating all bigons.
    
    The vertices $v, v'$ are connected by a segment $\gamma$ of a flow line since they lie on the intersection between a leaf $L^s$ of the stable lamination $\Lam^s$ and a leaf $L^u$ of the unstable lamination $\Lam^u$. We can exhibit three disks $A, B, E$ as follows:
    \[e^u\cup\gamma = \partial A, A\subset L^u;\]
    \[e^s\cup\gamma = \partial E, E\subset L^s_t;\]
    \[e^s\cup e^u = \partial B\]

    \begin{figure}
        \centering
        \includegraphics[width=0.6\linewidth]{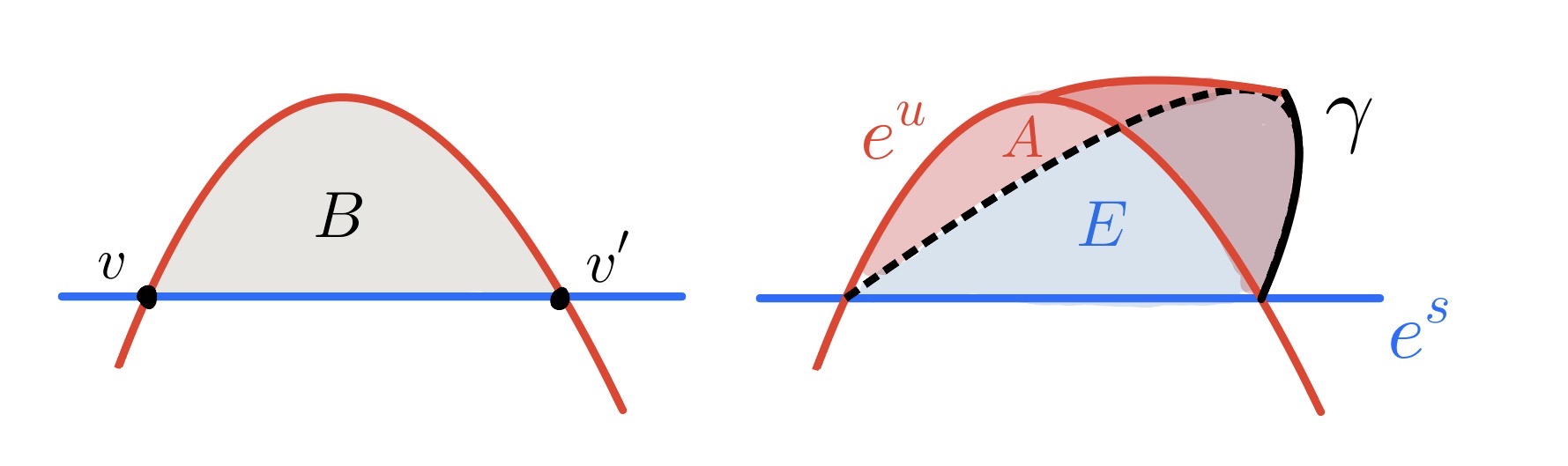}
        \caption{Left: A stable simple bigon $B$; Right: The flow line segment $\gamma$, two disks $A,E$ in an unstable and a stable leaf, and the isotopy ball bounded by $A,E,B$.}
        \label{fig: isotopy ball}
    \end{figure}

    Taken together, $A\cup B\cup E$ bounds a ball $P\subset M$. See Figure \ref{fig: isotopy ball} for an illustration of the isotopy ball. The desired isotopy is pushing $A$ along $E$ through $P$ such that $B$ is eliminated from the 2-simplex $\kappa$. From the perspective of $\kappa$, this is equivalent to sliding $e^u$ toward $e^s$ until they coincide as a flow segment, then pushing the flow segment off of $\kappa$ to the opposite side of $P$, eliminating $B$. Since the interior of the innermost bigon is free of $\Lam^s, \Lam^u$ and fixes $\Lam^s$ as a set, this isotopy maintains transversality $\Lam^s\pitchfork\Lam^u$ and preserves normality of $\Lam^s$.

\subsection*{Step 1: Normalizing $\Lam^u$ near the vertical boundary $\partial_vN(\Sigma^s)$} Using Theorem \ref{thm: Gabai kneser normal form} and Theorem \ref{thm: Gabai finite essential branched surface no tori carried}, ambiently isotope the stable lamination $\Lam^s$ to normal position carried by an essential branched surface $\Sigma^s$. In particular, since the manifold M is hyperbolic, $\Sigma^s$ carries no tori. At this moment, $\Lam^u$ is not in normal position, but it is still transverse to $\Lam^s$. Refine the triangulation by a subdivision so that in each 3-simplex, $\Lam^s$ has at most three normal disk types. Observe that this is a uniformly bounded subdivision. 

Normalizing $\Lam^s$ gives us a normal branched surface $\Sigma^s$ that carries $\Lam^s$. Consider the vertical boundary of the fibered neighborhood $\partial_vN(\Sigma^s)$. Next, we isotope $\Lam^u$ so that its intersection with the vertical boundary is a lamination by vertical arcs, as shown in Figure \ref{fig: no balloons}.

\begin{prop}
    There is an isotopy of $\Lam^u$ so that all intersections of flow lines on an annular stable leaf $A$ with the vertical boundary $\partial N_v(\Sigma^s)$ have the same sign.
\end{prop}

\begin{proof}
    Let $A^s$ be an annular stable leaf whose core $\gamma$ is a closed orbit. We will describe how to isotope the branch locus of the stable branched surface to be disjoint from $\gamma$. Note that to isotope the unstable lamination $\Lam^u$ as desired, we do the reverse isotopy. On each component of $A^s\setminus \gamma$, the flow foliation is a product $\R\times I$. Isotope the branch locus transverse to this product foliation and eliminate every innermost bigon bounded by a branch-locus arc and a flow segment. Each such isotopy is supported away from $\gamma$. After the bigon reductions, every remaining branch-locus arc is monotone with respect to the product foliation, and therefore all of its intersections with the flow lines have the same sign. See Figure \ref{fig: no balloons} for an illustration.
\end{proof}

Normalizing near the vertical boundary gives us the following useful corollaries.

\begin{corollary}\label{cor: no bigons in vertical boundary}
    In the vertical boundary $\partial_vN(\Sigma^s)$, there exists no bigons with one edge on a stable leaf and the other edge on an unstable leaf.
\end{corollary}

\begin{corollary}\label{cor: outside fibered neighborhood branched surface is annuli and spiral disk}
    Let $V$ be a connected component of the closed complementary region of $N(\Sigma^s)$. Then $\Lam^u\cap V$ consists of annuli and non-compact disks that spiral to the annuli.
\end{corollary}


\begin{figure}
    \centering
    \includegraphics[width=0.5\linewidth]{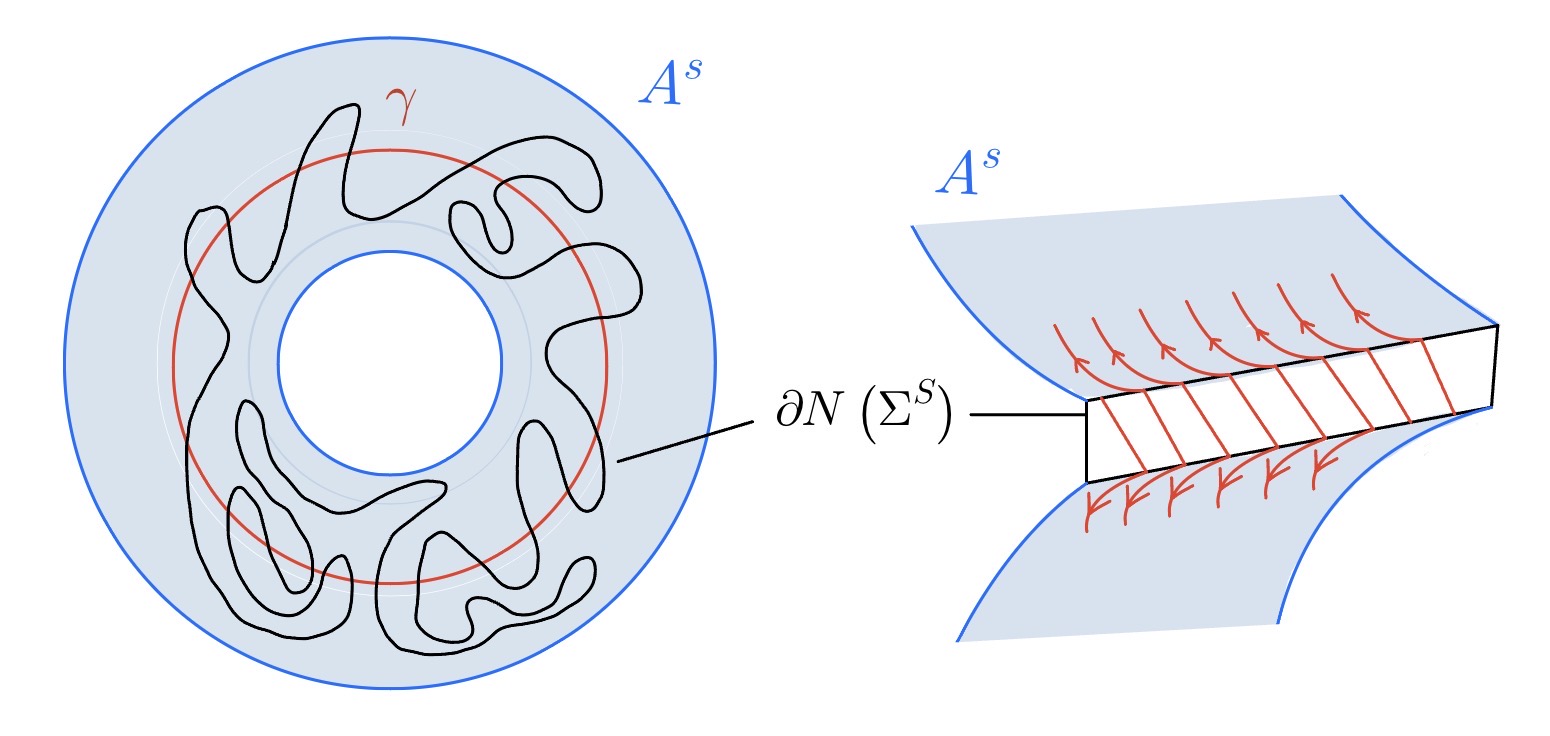}
    \caption{Left: an annular stable leaf bordering the splitting $A^s$ with its closed orbit $\gamma$ in red, and the intersection with the vertical boundary of $N_v(\Sigma^s)$ in black. Right: After the isotopy, all intersections between the flow lines and the vertical boundary have the same sign. In particular, there are no flow lines that enter and exit the vertical boundary on the same side.}
        \label{fig: no balloons}
\end{figure}

We proceed to normalize $\Lam^u$ with an eye towards maintaining transversality between the two laminations. 

\begin{definition}
    A local leaf of a lamination $\Lam$ with respect to a triangulation $\tau$ is a connected component of $\Lam\cap \sigma$, where $\sigma$ is a 3-simplex.
\end{definition}

\subsection*{Step 2: Normalizing near vertices of cellulated $\Sigma^s$}

Cellulate the stable branched surface $\Sigma^s$ so that every 0-complex is a branch point, every 1-complex is a branch arc, and every 2-complex is a sector. 

\begin{lemma}
    After an isotopy preserving the transverse intersection between $\Lam^s, \Lam^u$, there are no unstable monogons with ends in the fibered neighborhood of any vertex $v_i$ of the stable branched surface $\Sigma^s$.
\end{lemma}

\begin{proof}
    Let $\pi_s: N(\Sigma^s)\rightarrow \Sigma^s$ be the fiber projection map. Choose pairwise disjoint small saturated neighborhoods $U_i\subset N(\Sigma^s)$ for vertices $v_i$ of $\Sigma^s$. Near a vertex $v_i$, by small isotopies we can assume that the intersection 1-complex $\Sigma^s\cap \Sigma^s$ is disjoint from a neighborhood of $v_i$. Therefore, the fibered neighborhood $U_i$ of a small neighborhood of the vertex $v_i$ consists of disjoint stable and unstable sheets. Thus, in each $U_i$, the unstable lamination consists of finitely many clumps of sheets meeting the stable $I$-fibers transversely.

    We say that $\Lam^u$ has a \textit{returning clump} at vertex $v_i$ if there is a monogon $M$ created by an unstable clump passing through $U_i$ twice, so that $\partial M = \alpha\cup \beta$, where $\alpha\subset \Lam^u$, and $\beta$ is a subset of the $I$-fiber $\pi_s^{-1}(v_i)$. We say that the unstable carrier is \textit{strongly normalized} near $v_i$ if no such returning clump occurs. 

    Now normalize the vertices one at a time. Suppose $U_1,\ldots,U_{i-1}$ have already been strongly normalized. Consider $U_i$. Because the unstable carrier is finite, there are finitely many returning clumps in $U_i$. Choose one outermost returning clump that determines a monogon $M$ with boundary $\alpha\cup\beta$ as before. Then consider the intersection $M\cap \Lam^s$, which is a collection of arcs. 
    
    If any of the connected components have both vertices on $\beta$, we find a monogon in an essential branched surface $\Sigma^s$, which is impossible. If a component $c$ has both endpoints on $\alpha$, we see a bigon with edges in the stable and unstable laminations. Eliminate such bigons using the isotopy described in Section \ref{procedure: bigon reduction}. Now, the monogon $M$ can be eliminated by isotoping $\alpha$ across $M$ while maintaining transversality with $\Sigma^s$. This isotopy strictly decreases the number of unstable clumps in the vertex neighborhood $U_i$.

    Next, we argue that monogon elimination near a vertex $v_i$ preserves the strongly normalized vertices with lower index. To see this, consider that Corollary \ref{cor: no bigons in vertical boundary} implies any bigon reduction necessary in our monogon reduction step is performed in the fibered neighborhood $N(\Sigma^s)$. On every saturated vertex neighborhood $U_i$, the portion of the new unstable disk in is obtained from the old unstable disk by an isotopy through disks transverse to the stable $I$-fibers. Then the bigon reduction preserves the returning-clump pattern in $U_i$. In particular, if the unstable carrier was strongly normalized in $U_i$ before the reduction, it remains strongly normalized there afterward.

    Therefore, strongly normalizing each $U_i$ via monogon reduction terminates because every $U_i$ only has finitely many families of returning clumps, and there are finitely many vertices $v_i$. The resulting isotopy ensures that $\Lambda^u$ has no monogons near vertices $v_i$ of the cellulated $\Sigma^s$.
\end{proof}

\subsection*{Step 3: Normalizing near edges.}

Next, we normalize an edge whose endpoints have undergone the previous strong normalization.
\begin{lemma}
    After an isotopy preserving transversality of $\Lam^s, \Lam^u$, there are no returning monogons, boundary compressions, or generalized boundary compressions in the fibered neighborhood of any edges $e$ of $\Sigma^s$.
\end{lemma}

\begin{proof}
    Consider the fibered neighborhood $E_e = \pi^{-1}_s(e)$ of an edge $e\subset\Sigma^s$.

    Since there is no branch vertex in the interior of $e$, the local model of $\Sigma^s$ is constant along $e$. Cutting $E_e$ along the stable sector faces incident to $e$ produces finitely many rectangular product regions. On each rectangle, the stable and unstable laminations determine finite families of properly embedded arcs, grouped into parallel clumps.

    We say that $\Lambda^u$ is \textit{strongly normalized} near $e$ if these arc systems contain neither returning unstable clumps nor reducible stable–unstable bigons. We show that we can isotope $\Lam^u$ to be strongly normalized near every $e$ while preserving the strongly normalized vertex neighborhoods.

    Consider unstable returning clumps that create monogons with boundary $\alpha\cup \beta$ where $\beta$ lies in the stable $I$-bundle fiber in $E_e$ and $\alpha\subset\Lam^u$. Each monogon, as before, intersects the stable lamination $\Lam^s$ in arcs. Such arcs cannot have both endpoints in the fiber since $\Sigma^s$ is essential, and any arcs with endpoints on $\alpha$ lie in \textit{stable-unstable bigons}, which are bigons with one edge on $\Lam^s$ and one edge on $\Lam^u$, that can be eliminated via isotopy as in Section \ref{procedure: bigon reduction}. We remark that such bigon-reduction isotopies are not supported only in a neighborhood of $E_e$. However, as before, bigon-reduction occurs only in the fibered neighborhood $N(\Sigma^s)$ and respects previously strongly-normalized vertices and edges.

    Perform monogon reduction for all returning clumps of the unstable lamination. After finitely many isotopies, there are no returning unstable clumps in $E_e$. One way to see this finiteness is that each family of monogons and bigons corresponds to a local sector of the finite complexity unstable dynamic branched surface. Now, every unstable trace on the rectangular faces of the edge neighborhood joins distinct boundary intervals. We next eliminate stable–unstable bigons with one edge on a component of $\pi^{-1}_s(e)$ and one edge on $\Lam^u\cap E_e$. As before, there are only finitely many such families due to the underlying dynamic branched surfaces being finite, and each family of bigons may be eliminated with an isotopy supported in $N(\Sigma^s)$ following Section \ref{procedure: bigon reduction}.
\end{proof}

At this point of the normalization process, the stable branched surface near $e$ contains only finitely many rectangular faces and finitely many distinguished boundary intervals. Since there are no monogons, every unstable trace joins two allowed boundary intervals without returning to the same side. Parallel arcs with the same endpoint data are grouped into one clump.

Hence there are only finitely many nonparallel unstable arc types. Since each stable and unstable clump meets at most once, the total number of crossings is uniformly bounded. Therefore only finitely many planar combinatorial diagrams can occur on the rectangular faces.

The gluing of these rectangles around $e$ is determined by the local behavior of $\Sigma^s$. Hence only finitely many strongly normalized unstable local sheet types occur in $E_e$.

Repeating this over the finitely many branch edges gives strong normalization near the entire $1$-skeleton.

\subsection*{Step 4: Normalizing near a sector}

Let $D$ be a sector of $\Sigma^s$. After subdividing if necessary, assume that $D$ is a disk. By the previous two steps, $\Lam^u$ has been strongly normalized near $\partial D$, with only finitely many possible local boundary types for each edge $e$. Away from a collar $C(\partial D)\subset D$ that has been strongly normalized, the disk $D_0 =\overline{D\setminus C(\partial D)}$ is disjoint from any branching, and it has fibered neighborhood $\pi^{-1}_s(D_0)$ that is topologically $D_0\times I\cong B^3$. 

Normalize $\Lam^u\cap \pi^{-1}_s(D_0)$ using the same technique as in Gabai's normalization process, performing compressions and boundary compressions and eliminating bigons as necessary. As before, for any compression of boundary compression disk $\Delta$, consider the intersection with the stable sector $\Sigma^s$ near $D_0$, which is a family of arc and circles because $\Sigma^s$ is unbranched near $D_0$. Circles can be removed by innermost disk arguments, and arcs with endpoints on an unstable leaf can be eliminated by a bigon elimination isotopy. Upon termination, such isotopies put $\Lam^u$ into normal position with respect to $\Sigma^s$ while preserving the normalized 1-skeleton of $\Sigma^s$.

Apply Kneser finiteness in the fixed product ball $\pi^{-1}_s(D_0)$. The boundary pattern on belongs to a finite collection, so only finitely many nonparallel normalized surface types occur in $\pi^{-1}_s(D_0)$. Amalgamating parallel copies into clumps therefore gives only finitely many possible unstable local sheet types over $D$. Thus we need only finitely many unstable branched surfaces restricted to $N(\Sigma^s)$.

To show that the stable and unstable laminations can be simultaneously normalized, it remains to finiteness of branched surfaces carrying every complementary region $V\subset\overline{M-N(\Sigma^s)}$. We use Theorem \ref{thm: finite singular orbits} of the second author \cite{LiSingularOrbit} to establish that there are only uniformly finitely many possibilities for the complementary regions in the pseudo-Anosov case. 

By Corollary \ref{cor: outside fibered neighborhood branched surface is annuli and spiral disk}, $\Lambda^u\cap V$ consists of annular leaves and noncompact disk leaves spiraling to annular ones. Each of the finitely many boundary carrier patterns on $\partial V$ extends to only finitely many carriers in $V$. Hence the inside and outside carriers glue along their common boundary to give one of finitely many global branched surfaces $\Sigma^u_1,\ldots,\Sigma^u_N$ carrying $\Lam^u$.

Finally put the outside pieces into relative normal form as well. The resulting global carriers are Kneser branched surfaces. If a subdivision was used anywhere in the construction, take one common finite refinement for the finite collection $\Sigma^s,\Sigma^u_1,\ldots,\Sigma^u_N$.

Thus both $\Lambda^s$ and $\Lambda^u$ are carried by normal branched surfaces for the same triangulation or uniformly finite refinement. Therefore the stable and unstable laminations are simultaneously normalized. This concludes the proof of Theorem \ref{simnorm}.\qed

\subsection*{Finiteness of Kneser branched surfaces}

We record the following consequences of Theorem \ref{simnorm}.

\begin{theorem}\label{thm: finite pairs of Kneser branched surfaces}
    Let $M$ be a closed 3-manifold. There are finitely many pairs of essential branched surfaces $B = \{(\Sigma^s_i, \Sigma^u_i)\}$ so that for any Anosov or pseudo-Anosov flow $\phi$ on $M$, there is a pair $(\Sigma^s_i, \Sigma^u_i)\in B$ carrying the stable and unstable laminations $\Lam^s, \Lam^u$, respectively, up to ambient isotopy. 
\end{theorem}

\begin{proof}
    First perform an isotopy to simultaneously normalize $\Lam^u, \Lam^s$ to the triangulation $\tau$. Following Gabai's paper, there are finitely many essential branched surfaces carrying any normalized lamination, so there are also finitely many pairs that carry $\Lam^u, \Lam^s$.
\end{proof}

\begin{definition}
    Let $\phi$ be a pseudo-Anosov flow on $M$. We call the pair of essential branched surfaces $\Sigma^s, \Sigma^u$ carrying the stable and unstable laminations in normal position their \textit{Kneser branched surfaces}.
\end{definition}

\section{Parallel Locus}\label{section: parallel locus}

From Theorem \ref{thm: finite pairs of Kneser branched surfaces}, we conclude that there are finitely pairs $\Sigma^s, \Sigma^u$ of possible Kneser branched surfaces. From now on, we consider that when there is an infinite sequence of Anosov or pseudo-Anosov flows on $M$, all of their stable laminations (resp. unstable laminations) are carried by the same Kneser branched surface. 

In general, the stable and unstable laminations must be transverse to each other. However, from the point of view of the triangulation, the Kneser branched surfaces may occupy parallel sheets of the same normal disk type, making them seem tangent over several 3-simplices. In this section, we made some definitions to describe how tangent the stable and unstable laminations might seem to each other. 

\begin{definition}[Parallel Locus]
    Let $\Lams, \Lamu$ be stable and unstable laminations of a pA flow $\phi$ on $M$ carried by Kneser branched surfaces $\Sigma^u, \Sigma^s$ with fiber neighborhoods $N(\Sigma^u), N(\Sigma^s)$. In the universal cover, we have $\cLams, \cLamu$ carried by branched surfaces $\cSigu, \cSigs$ with fiber neighborhoods $\widetilde{N}(\cSigu), \widetilde{N}(\cSigs)$. Let the fiber at a point $x$ of a stable branched surface $\Sigma^s$ be $I_x(\Sigma^s)$, or we write $I_x^s$ when there is no ambiguity on the branched surface. Analogously the fiber of the unstable one $I_x(\Sigma^u)$, or $I_x^u$ when it's not ambiguous. For a subset $X\subset \widetilde{L}^s$, its \textit{stable saturation} with respect to $\cSigs$, denoted $\sat(X,\cSigs)$, is the preimage $\widetilde{\pi}^{-1}_s(X)\subset \widetilde{N}(\cSigs)$, where $\widetilde{\pi}: \widetilde{N}(\cSigs)\rightarrow \cSigs$ is the projection map. When there is no ambiguity about the branched surface, we abbreviate to $\sat_s(X)$. The \textit{unstable saturation} $\sat(X, \cSigu)$, or $\sat_u(X)$, is defined analogously.

Let a point in a stable leaf $x\in \widetilde{L}^s$ be in the \textit{$R$-parallel locus} $P_R(\widetilde{L}^s, \widetilde{L}^u; \cSigu)$ with respect to an unstable leaf $\widetilde{L}^u$ as well as the unstable branched surfaces $\cSigu$ if there is a radius-$R$ stable disk $D_R^s\subset \widetilde{L}^s$ containing $x$ and an unstable $R$-disk $D^u\subset\widetilde{L}_1^u$ for some $\widetilde{L}_1^u$ carried by $\widetilde{\Sigma^u}$ such that $D_R^s$ is contained in the unstable saturation of $D^u$ as a section of the unstable $I$-bundle. More precisely,
\begin{itemize}
        \item The unstable saturation of $D^u$ contains the stable disk: $D^s_R\subset \sat_u(D^u)$.
        \item The unstable fiber projection $\pi_u:\sat(D^u)\rightarrow D^u$ restricts to a homeomorphism $\pi_u|_{D_R^s}:D_R^s\rightarrow D^u$.
    \end{itemize}  

If $D_R^s, D^u$ satisfy the conditions above, we call them \textit{dual $R$-disks}. In particular, we call any pair of topological leaf-wise disks in a stable and unstable leaf satisfying the above conditions $\textit{dual}$ to each other.

    \end{definition}

 Heuristically, the $R$-parallel locus $P_R(\widetilde{L}^s, \widetilde{L}^u; \cSigu)$ measures how much of a stable leaf is close to a neighborhood of an unstable leaf to radius $R$. Its size is determined by the branched surface $\cSigu$ and their fibered neighborhoods. The definition of a parallel locus is equivariant under deck transformations. 

 
\begin{definition}[Shadow]
    The $R$-\textit{shadow} of a branched surface $ \cSigu$ on a stable leaf $\widetilde{L}^s$ is the union of parallel loci on $\widetilde{L}^s$ of all scales less than $R$, taken union over all leaves $\widetilde{L}^u\subset \cLamu$. The shadow is a subset of $\widetilde{L}^s$. In other words, 
    \[P_R(\widetilde{L}^s, \cSigu) = \bigcup_{\widetilde{L}^u\subset \cSigs} \bigcup_{0<r<R}P_r(\widetilde{L}^s, \widetilde{L}^u; \cSigu).\]
\end{definition}

Heuristically, the shadow measures how much of a stable leaf is close to any unstable leaf, regardless of the scale. In particular, it tells us that if a point $p$ in a stable leaf is not inside the shadow of some unstable branched surface, then the stable fiber $I_p$ is either disjoint from that branched surface or that a neighborhood of $p$ coincides with a set of unstable fibers of the unstable branched surface.

 Since parallel loci depend on choosing stable and unstable branched surfaces $\cSigs, \cSigu$ carrying $\cLams, \cLamu$, we must discuss how splitting a branched surface would influence parallel loci and shadows. To this end, we first define splittings of branched surfaces to some distance, following the conventions of Gabai-Oertel \cite{GabaiOertel}.

\begin{definition}Let $\Sigma, \Sigma'$ be branched surface embedded in $M$. We say that $\Sigma'$ is a \textit{splitting} of $\Sigma$ if the following condition holds. There is an $I$-bundle $J\subset M$ such that $N(\Sigma)\cup J$ where $J\cap N(\Sigma')\subset \partial ; \partial_hJ\subset \partial_hN(\Sigma')$; and $\partial_v J\cap N(\Sigma')\subset \partial_vN(\Sigma')$ has finitely many components whose fibers are fibers of $\partial_vN(\Sigma')$. In practice, the splittings we use respect a lamination fully carried by $\Sigma$. 

We say that the splitting is to distance $R$ if the distance between the old and new vertical boundaries measured in $N(\Sigma)\setminus N(\Sigma')$ is between $R$ and $2R$.
    
\end{definition}

\begin{prop}
    There are only finitely many distance-$R$ splittings of a branched surface $\Sigma$ in a closed 3-manifold $M$.\qed
\end{prop}

\section{Arbitrarily Large Parallel Loci}\label{section: swaths}

In this section, we prove the following result. 
\begin{theorem}\label{noswaths}
    Let $M$ be a closed hyperbolic 3-manifold. Then there exists a uniform constant $R\in\N$ such that for any pA flow $\phi$ on $M$, the $R$-parallel locus $P_R(\widetilde{L}^s, \widetilde{L}^u; \cSigs)$ is the empty set for any two leaves $\widetilde{L}^s, \widetilde{L}^u\subset \cLams, \cLamu$. 
\end{theorem}

Assume for the sake of contradiction that there is an infinite sequence $\{\phi_i\}$ of non orbit equivalent pA flows with stable and unstable laminations $\Lams_i, \Lamu_i$ containing leaves realizing arbitrarily large parallel locus. We show that these parallel regions propagate to limit objects.

\begin{prop}\label{prop: stable and unstable limits share sublamination}
        Let $\{\phi_i\}$ be a sequence of flows with stable and unstable laminations $\Lams_i, \Lamu_i$ carried by branched surfaces $\cSigs, \cSigu$ so that, for any $R\in\N$, there exists some $i$ and leaves $\widetilde{L}_i^s, \widetilde{L}_i^u\subset \cLams_i, \cLamu_i$ with nonempty $R$-parallel locus $P_R(\widetilde{L}_i^s, \widetilde{L}_i^u; \cSigs, \cSigu)$. Then there is a subsequence of pairs $\Lam^s_{i_k}, \Lam^u_{i_k} $ such that the limit of stable laminations $\Lam^s_\text{lim}$ and the limit of unstable laminations $\Lam^u_\text{lim}$ share a nonempty common sublamination $\Lam_\text{lim}$.

\end{prop}
\begin{proof}
    First, up to taking a subsequence, Lemma \ref{lemma: laminations converge to limit} implies that the stable laminations converge to a limit $\Lam^s_\text{lim}$, and the unstable laminations converge to a limit $\Lam^u_\text{lim}$.

    One way to proceed is to invoke Lemma \ref{lemma: disk converge to lam} and conclude that the two limit laminations contain the same closure of a leaf build from a limit of parallel disks of increasing radius, up to taking a normal isotopy after taking limits. The closure of these leafs is a common sublamination. To see this, we record the following proof. 

    For every $R\in\N$, consider the radius-$R$ disk $D^s_R(x)\subset \widetilde{L}_{i(R)}^s, D^u_R(y)\subset \widetilde{L}_{i(R)}^u$ given by the nonempty $R$-parallel locus. They track each other as the same normal disk types across a radius-$R$ region. Since all the stable branched surfaces are carried by the same stable branched surface $\cSigs$, and $\Sigma^s$ has only finitely many sectors, we can choose a subsequence of the flows so that each stable disk $D_R^s$ is centered at the same point. 

    For each $R$, we have a radius-$R$ region in both the stable and unstable leaves tracking parallel normal disks of the same type. Consider the sequence of pairs of radius-$R$ disks in the stable and unstable laminations. Observe that as the disks grow in radius, there is ambiguity only at branch loci. For any finite $R$, the radius-$R$ disk intersects finitely many branch loci arcs. By the pigeonhole principle, each time we encounter an arc of the branch locus, since there are only finitely many sector meeting at each branch locus, we can take a subsequence so that all the disks continue into the same sector after branching. Continue to take subsequences, so that in the limit, these disks limit to a leaf persisting in both the stable and unstable limits.
\end{proof}

Upon finding a common sublamination $\Lam_{\lim}$ of the stable and unstable limits, replace it with a minimal sublamination of itself to ensure that every leaf is dense. There are two possibilities: either this new lamination has no holonomy, or there is a curve $\gamma$ that exhibits nontrivial holonomy. In the next two subsections, we discuss each case and reach a contradiction in each case, thereby proving Theorem \ref{noswaths}.

\subsection{Euler class if there is no holonomy}\label{subsection: no holonomy euler class} In this subsection, we show that if the limit lamination $\Lam_{\lim}$ has no holonomy, then we can find two homotopic approximations to its tangent field that have different Euler classes, thus reaching a contradiction. The proof of this case is highly similar to the proof in Section \ref{subsection: anosov no holonomy}, and we only point out the differences.

First consider the sequence of stable laminations $\Lam^s_i$ converging to $\Lam_{\lim}$. We prove the following result in parallel with Lemma \ref{lemma: extend lim lam tangent to neighborhood}.

\begin{lemma}
    For a sufficiently small neighborhood $U$ of $\Lam_{\lim}$, there is an extension $\xi_i$ of the 2-plane field $T\Lam^s_i$ to $U$ so that the Euler class $e(\xi_i)=0$.
\end{lemma}
\begin{proof}
    Up to taking a branched cover over the singular orbits, we can assume that the stable lamination is co-orientable. Then, fill the complementary region for each $\Lam^s_i$ with the "stack of chairs" or monkey saddles construction. See for example Example 4.18 in Calegari's book \cite{foliationbook}. We obtain a foliation $\mathcal{F}_i$ for each $i$. Then, we take the restriction of the tangent bundle $\xi_i = T\mathcal{F}_i|_{U}$.

    As before, the tangent bundles $T\Lam^s_i$ over the stable laminations all admit no-where vanishing sections given by the flow direction. We need to extend the sections to the neighborhood $U$. Another way to think about the non-vanishing section on $T\Lam_i^s$ is that we obtain it at any point $x$ by intersecting the tangent plane of the stable lamination with the tangent plant of the unstable lamination. After stuffing the complementary regions of $\Lam^s_i$, we stuff the interstitial part of the complementary regions of the unstable laminations $\Lam^u_i$ so that we obtain a foliation $\mathcal{F}^u_i$ outside the guts, which can be made a sufficiently small compact neighborhood of the singular orbits. In particular, the guts can be made to be disjoint from $U$. By construction, the foliations $\mathcal{F}_i, \mathcal{F}^u_i$ are transverse wherever they are both defined. For any point $x\in U$, consider the intersection $T_x\mathcal{F}^u_i\cap T_x\mathcal{F}_i$, which agrees with the flow direction on $T\Lam^s_i$ where defined and gives a non-vanishing section of the bundle $\xi_i$ on $U$. 
    
    Having found a non-vanishing section of $\xi_i$ on $U$, we conclude that $e(\xi_i)=0$. 
\end{proof}

Next, extend the tangent bundle $T\Lam_{\lim}$ to a small neighborhood $U$, and call the bundle $\xi_{\lim}$. As in Section \ref{subsection: anosov no holonomy}, we can arrange for $\xi_{\lim}$ be to homotopic to $\xi_i$ for all sufficiently large $i$ using the convergence of tangent fields. Since $e(\xi_i)=0$, we conclude that $e(\xi_{\lim})=0$.

Now we construct another approximation to the tangent field $T\Lam_{\lim}$.

If $\Lam_{\lim}$ has no holonomy, there exists a transverse measure on it. An invariant measure gives a real-valued solution $\vec{x}$ so the branched surface equations $Ax=0$ given by $\Sigma^s$ is satisfied. There exists a sequence of rational approximations $\vec{x}_i$ satisfying the branch equations, so that $A\vec{x}_i=0$. Use the construction after Lemma \ref{lemma: construct B epsilon} to construct approximating closed smooth surfaces $S_i$ realizing the branched surface equations for each $i$. We conclude as before that after some homotopy in the neighborhood $U$, the bundle $TS_i$ is the pullback bundle $i^*\xi_{\lim}$ where $i: S_i\hookrightarrow U$ is the inclusion map. Since $\Sigma^s$ and its splittings carry no spheres or tori, the surfaces $S_i$ are of genus at least $2$, so the Euler class $e(\xi_{\lim})\neq 0$.

\subsection{When there is holonomy}\label{subsection: holonomy swaths}

Next we consider the case where there is holonomy in the shared limit laminations between the limit of the stable laminations and the limit of the unstable laminations. There is an oriented curve $\gamma_0$ in a leaf $L\subset\Lam_\text{lim}$ such that $\Lam_\text{lim}$ has holonomy along $\gamma_0$. 

\begin{prop}\label{prop: get perfect fits}
    Let $\Lam^s_i, \Lam_i^u$ be sequences of stable and unstable laminations of Anosov or pseudo-Anosov flows. Assume as before that both sequences converge to limit laminations $\Lam_{\lim}^s, \Lam_{\lim}^u$, and that the limits share a sublamination $\Lam_{\lim}$. If $\Lam_{\lim}$ has holonomy, there exists $N_1\in\N$ so that for $i>N_1$, there are closed curves $\gamma^s_i, \gamma^u_i$ in $\Lam^s_i, \Lam^u_i$ respectively with the following properties:
    \begin{itemize}
        \item All $\gamma^s_i$ represent the same isotopy class in $M$: $[\gamma^s_i]= [\gamma^s_j]$;
        \item All $\gamma^u_i$ represent the same isotopy class: $[\gamma^u_i]= [\gamma^u_j]$;
        \item The curves $\gamma^s_i, \gamma^u_i$ are anti-homotopic in $M$: $[\gamma^s_i] =-[\gamma^u_i]$.
    \end{itemize} 
\end{prop}
\begin{proof}
    By definition, we can find a transversal $I_0$ to $\gamma_0$ so that the holonomy transport gives a map $f_{\gamma_0}:I_0\rightarrow I_0$ satisfying $f_{\gamma_0}(I_0)\subset I_0$. In fact, we can find a fixed point $f_{\gamma_0}(p)=p$ such that the closed curve obtained by parallel transport of $p$ has one-sided contracting holonomy.

    In the limit lamination $\Lam_{\lim}$, let $L$ be the leaf containing $\gamma_0$, and let $I_1$ be the component of $I_0\setminus f_{\gamma_0}(I_0)$ on the contracting side of $\gamma_0$. By minimality of $\Lam_{\lim}$, $L$ meets the open transversal $I_1$. There is a leafwise path $\gamma_1\subset L$ connecting $p_0\in\gamma_0\cap I_0$ to a point $p_1\in I_1$. Consider a smaller transversal $I_2\subset I_0$ such that the parallel transport $f_{\gamma_1}:I_2\rightarrow I_1$ satisfies $f_{\gamma_1}(I_2)\subset I_1$. See Figure \ref{fig: holonomy spinning} for an illustration. 

    \begin{figure}
        \centering
        \includegraphics[width=0.65\linewidth]{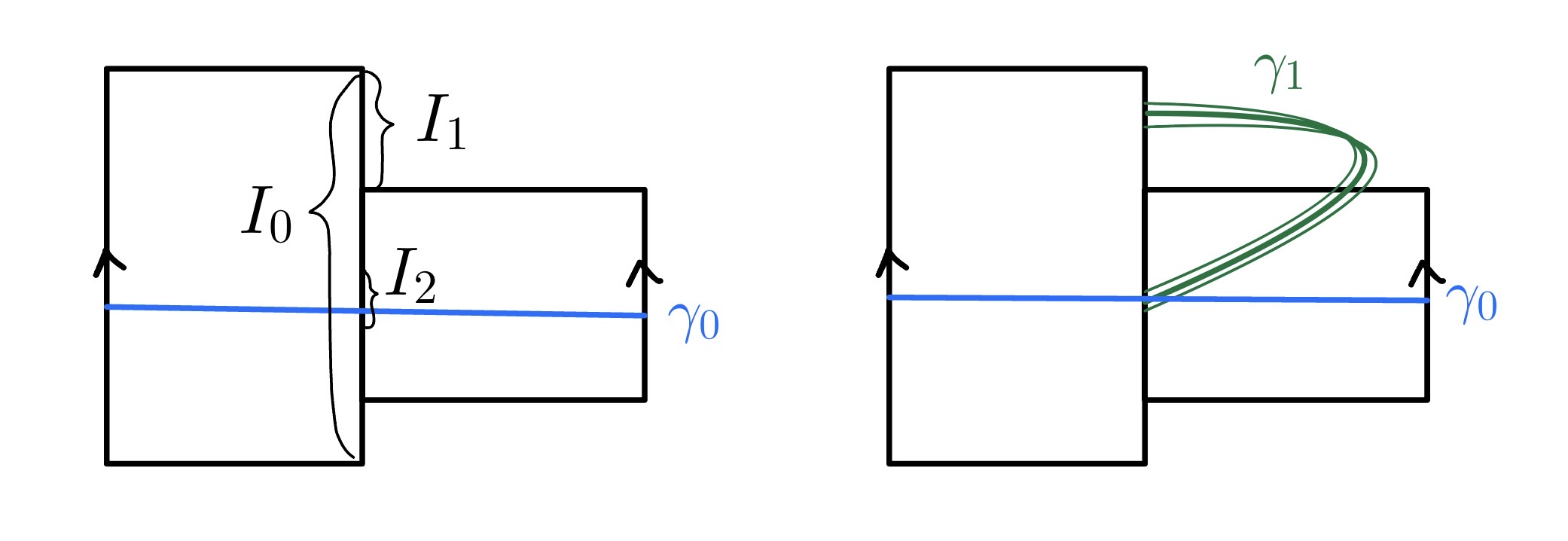}
        \caption{The left figure illustrates $\gamma_0$ with one-sided contracting holonomy, where the left and right edges of the picture are identified. The right figure shows a curve $\gamma_1$ in a leaf so that its holonomy map maps $I_2$ to a subset of $I_1$.}
        \label{fig: holonomy spinning}
    \end{figure}
    
    In the sequence $\{\Lam_i^s\}$, there exists $N_1\in\N$ so that for $i>N_1$, we can perform an ambient isotopy so that $\gamma_1\subset \Lam^s_i$. Choose a compact tubular neighborhood $U$ of $\gamma_1$ and arrange the isotopies to be $C^1$ small and supported on $U$. Then, after possibly shrinking $I_2$ so that the closure of $f_{\gamma_1}(I_2)$ is contained in the interior of $I_1$, the holonomy maps $f_{\gamma_1}^{s,i}:I_2^{}\rightarrow I_1$ satisfy $f_{\gamma_1}^{s,i}(I_2)\subset I_1$. 

    Since $\gamma_0$ in the limit has one-sided contracting holonomy, and $I_1, I_2$ lie on the contracting side, there exists $k\in \N$ so that an iterate $f_{\gamma_0}^k(I_1)\subset I_2$. This means $$F = f_{\gamma_0}^k\circ f_{\gamma_1}:I_2\rightarrow I_2$$ satisfies $F(I_2)\subset I_2$. For sufficiently large i, after the small ambient isotopies identifying the same curves $\gamma_0, \gamma_1$ with curves lying in $\Lam_i^s$, the corresponding holonomy maps satisfy $(f_{\gamma_0}^{s,i})^k(I_1)\subset I_2$ for $i>N_2$. For $i>\max(N_1, N_2)$, we have the holonomy transport $$F^s_i=(f_{\gamma_0}^{s,i})^k\circ f_{\gamma_1}^{s,i}:I_2\rightarrow I_2$$ satisfying $F_i^s(I_2)\subset I_2$. Since $I_2$ is compact, by the Brouwer fixed point theorem, there is a point $p_i^s\in I_2$ with $F_i^s(p_i^s)=p_i^s$, and the parallel transport along $p_i^s$ yields a closed curve $\gamma^s_i\subset\Lam^s_i$ for each $i$. If the fixed point is not unique, choose the maximum fixed point. The curve $\gamma^s_i$ first follows $\gamma_1$ and then winds around a local copy of $\gamma_0$ $k$ times, so all $\gamma_i^s$ are of the same isotopy class. We have found curves $\gamma_i^s\subset \Lam^s_i$ such that $[\gamma_i^s]=[\gamma_j^s]$ for all sufficiently large $i,j>\max(N_1, N_2)$.

    Using the same argument on the unstable sequence $\Lam^u_i$ yields a closed curve $\gamma^u_i$ on $\Lam^u_i$ of the same homotopy class as $\beta$ for every sufficiently large $i$. Along a stable leaf of an Anosov/pseudo-Anosov flow, transverse holonomy is contracting in the forward flow direction; along an unstable leaf, transverse holonomy is contracting in the backward flow direction. Since the return map is chosen to have one-sided contracting holonomy, the orientation of the closed curve induced by the stable flow agreed with the direction in which $\beta$ was traversed, while the orientation induced by the unstable flow is opposite. Therefore, for sufficiently large $i$, there are closed oriented curves $\gamma^s_i, \gamma^u_i$ in $\Lam^s_i, \Lam^u_i$ respectively that are anti-homotopic in $M$. In other words, $[\gamma^s_i]=-[\gamma^u_i]$.
\end{proof}

The curves $\gamma^s_i, \gamma^u_i$ constitute a pair of \textit{perfect fits}. Next, we show using the geometry near the annular leaves $A^s, A^u$ containing $\gamma^s_i, \gamma^u_i$ in $\Lam_i^u, \Lam^s_i$ that these perfect fits should not exist.

Fix some $i>\max(N_1, N_2)$, and let $\gamma^s_i$ be the closed curve in $\Lam^s_i$ obtained from the maximum of the fixed point set $x^s_i=\max\{\text{Fix}(F^s_i)\}$. Similarly, choose $\gamma^u_i$ to be obtained from the maximum $x^u_i=\max\{\text{Fix}(F^u_i)\}$. Call the leaves corresponding to $\gamma^s_i, \gamma^u_i$ the topmost annular leaves of the stable and unstable laminations intersecting $I_2$. Assume that $x^u_i>x^s_i$. The argument is analogous for the other case.


\begin{prop}[Common fence construction]\label{prop: common fence}
    Let $A_i^s$ and $A_i^u$ be stable and unstable annular leaves. That there is a common ambient transversal $J$ meeting both annuli and induced holonomy maps $H_i^s,H_i^u:J\longrightarrow J$ such that, for $\epsilon>0$, for all sufficiently large $i$,
\[
H_i^s(J)\subset\text{int}(J),H_i^u(J)\subset\text{int}(J)
\]
\[
\|H_i^s-H_i^u\|_{C^0(J)}\leq \epsilon
\]
Each one-step return to $J$ makes positive progress toward a chosen end of the corresponding annular leaf.

Further, for all sufficiently large $i$, there exists a one-sided common fence $\mathcal F_i$ based at $J$ which contains both the stable and unstable holonomy continuations and extends arbitrarily far toward the chosen ends of $A_i^s$ and $A_i^u$. Moreover, the transverse width of $\mathcal F_i$ may be bounded uniformly along the entire fence. In particular, the width of $\mathcal F_i$ tends to zero as $i\to\infty$.

\end{prop}
\begin{proof}
    Assume $\gamma^s_i$ to be the closed curve obtained from the maximum fixed point $x^s_i$ of $F_i^s$, and $\gamma^u_i$ obtained from the maximum fixed point $x^u_i$ of $F_i^u$. The proof is similar to the proof of Proposition \ref{prop: get perfect fits}.

    Consider the limit lamination $\Lam_{\lim}$, and find the maximum fixed point $p$ of the holonomy map $F:I_2\rightarrow I_2$. We obtain from $p$ a closed curve $\gamma_{\lim}\subset \Lam_{\lim}$ with one-sided contracting holonomy. Let $I_3$ be the connected component of $I_2\setminus F(I_2)$ on the contracting side of $\gamma_{\lim}$. Since the stable and unstable closed curves $\gamma^u_i, \gamma^s_i$ arise as small transverse push-offs of the same limiting return curve $\gamma_{\lim}$, their fixed points $x_i^s, x_i^u$ converge to $p$. Take a transversal $J\subset \text{int}(I_2)$ containing both $x_i^s, x_i^u$. There are leafwise paths $\gamma_2^{(s,i)}\subset A^s_i$ from $x_i^s$ to some $p^s\in I_3$ and $\gamma_2^{(u,i)}\subset A^u_i$ from $x_i^u$ to some $p^u\in I_3$. These paths can both be made perturbations of a path from $p$ to a point in $I_3$, and we have two holonomy maps 
    $$f_{\gamma_2^{(s,i)}}, f_{\gamma_2^{(u,i)}}:J\rightarrow \text{int}(I_3).$$
    
    Since $\gamma_i^s, \gamma_i^u$ have one-sided contracting holonomy, there is some $k>0$ for which $F^k(I_3)\subset \text{int}(J)$ for some $k>0$. Recall that both the sequences of holonomy maps $F^s_i$ and $F_i^u$ converge to the holonomy map $F$ in the limit $F$ uniformly on a compact subset $J$. In particular, for every $\epsilon>0$, there is some $N_1\in\N$ so that $i>N_1$ implies
\[\sup_{x\in J}\left|(F^s_i)^k(x)-(F^u_i)^k(x)\right|<\sup_{x\in J}\left|(F^s_i)^k(x)-F^k(x)\right| + \sup_{x\in J}\left|(F^u_i)^k(x)-F^k(x)\right|<\epsilon.\]

    Choose $\epsilon$ to be smaller than the margin $\frac{d(F^k(J),\partial J)}{10}$. Therefore, for sufficently large $i$, we conclude that $(F_i^s)^k(I_3)\subset\text{int}(J), (F_i^u)^k(I_3)\subset\text{int}(J)$.




Define holonomy maps 
\[H_i^s = (F_i^s)^k\circ f_{\gamma_2^{(s,i)}}: J\rightarrow \text{int}(J)\]
\[H_i^u = (F_i^u)^k\circ f_{\gamma_2^{(u,i)}}: J\rightarrow \text{int}(J)\]

As before, given $\epsilon>0$, for sufficiently large $i$, we obtain
\[
\|H_i^s-H_i^u\|_{C^0(J)}\leq \epsilon.
\]

\begin{figure}
    \centering
    \includegraphics[width=0.7\linewidth]{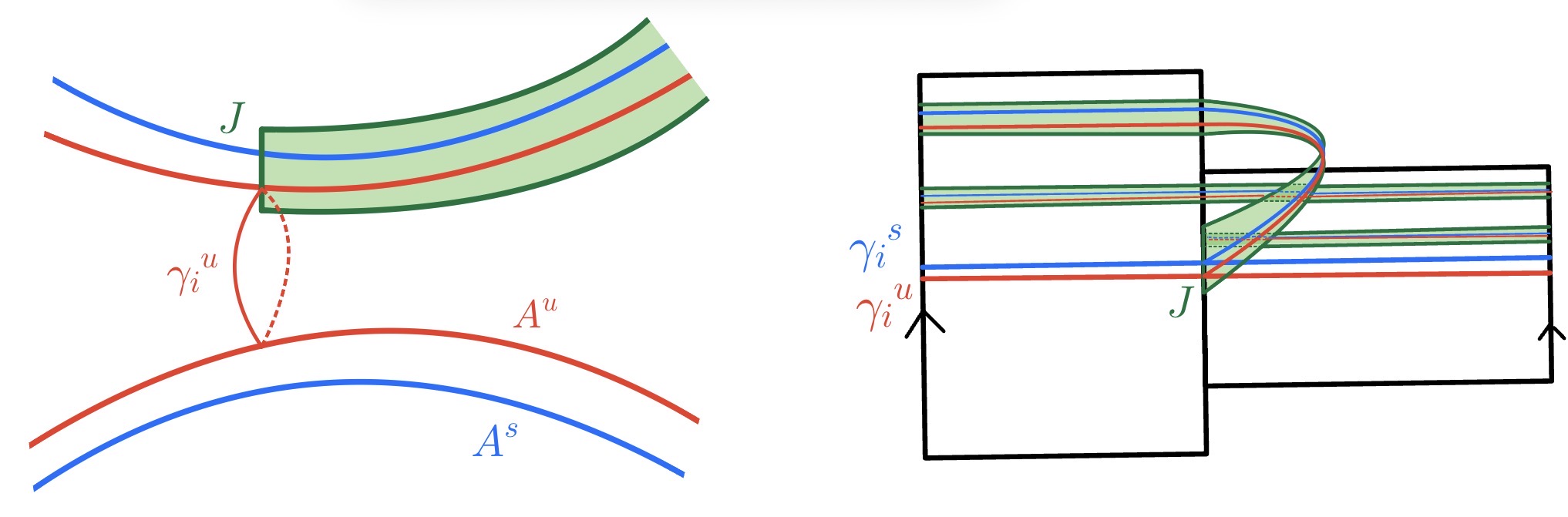}
    \caption{Left: An illustration of the fence $\mathcal{F}_i$ escaping to ends of the annular leaves $A^u, A^s$; Right: the construction of the fence after one iteration of the algorithm.}
    \label{fig: green fence}
\end{figure}

Each application of $H^s_i, H^u_i: J\rightarrow \text{int}(J)$ corresponds to winding $k$ times around $\gamma_i^s$ after going through $\gamma^s_i, \gamma^u_i$. Each time, we obtain stable and unstable transverse strips close together and ending inside $J$. On the first iteration, construct the fence $\mathcal{F}_i$ to start at $J$ and include both transverse strips, and repeat this construction to create the next segment of the fence. The fence will travel arbitrarily far into an end of each of the annular leaves. See Figure \ref{fig: green fence} for an illustration of the first iteration of creating the fence. Most importantly, the fence $\mathcal{F}_i$ has the property that any leaf meeting $J$ will continue to stay inside the fence in a specified direction. This direction is specified by which side of $\gamma_i^s, \gamma_i^u$ the path $\gamma_2^{(s,i)}, \gamma_2^{(u,i)}$ lie on. Since $\gamma_2^{(s,i)}, \gamma_2^{(u,i)}$ are 2-sided curves of their respective annular leaves, and both ends of the annular leaves $A^u, A^s$ are dense, we have two choices for a direction to travel in, and the choices will give fences to opposite ends of the annuli $A^s, A^u$.


\end{proof}

Next, fixe some sufficiently large $i$, and suppress the index $i$ for the rest of this section. We consider the geometry of the stable and unstable annular leaves $A^s, A^u$. Again assume that the fixed points satisfy $x^u>x^s$. 

\begin{definition}
    Suppose $A^s, A^u$ are annular leaves realizing a pair of perfect fits and containing a pair of closed curves $\gamma^s, \gamma^u$. Suppose further as above that $A^s, A^u$ are topmost with respect to $I_2$ and therefore $J$. We call the unique annular unstable leaf intersecting $A^s$ at a closed orbit $c_1$ the \textit{dual} of $A^s$ and denote it $B^u$. Similarly, let $B^s$ be the the \textit{dual} of $A^u$ as the unique annular stable leaf intersecting $A^u$ at a closed orbit $\gamma^u$. We have $A^s\cap B^u = \gamma^s, A^u\cap B^s = \gamma^u$, where $\gamma^s, \gamma^u$ are orbits of the flow.
\end{definition}

Since we have fences going to either end of $A^u, A^s$, without loss of generality, we observe that $\gamma^s=A^s\cap B^u$ has nontrivial intersection with a chosen fence $\mathcal{F}$. 

We discuss three cases and obtain a contradiction in each case to show that the geometry imposed by the fence $\mathcal{F}$ contradicts the topology imposed by the flow.

\subsubsection*{Case 1} Suppose that $B^u$ crosses $A^s$ in the orientation as pictured in Figure \ref{fig: swath contradiction case 1}. In the fence, $B^u$ cannot cross $A^u$, so it must intersect $J$ at a point $x'$. 

\begin{figure}
    \centering
    \includegraphics[width=0.7\linewidth]{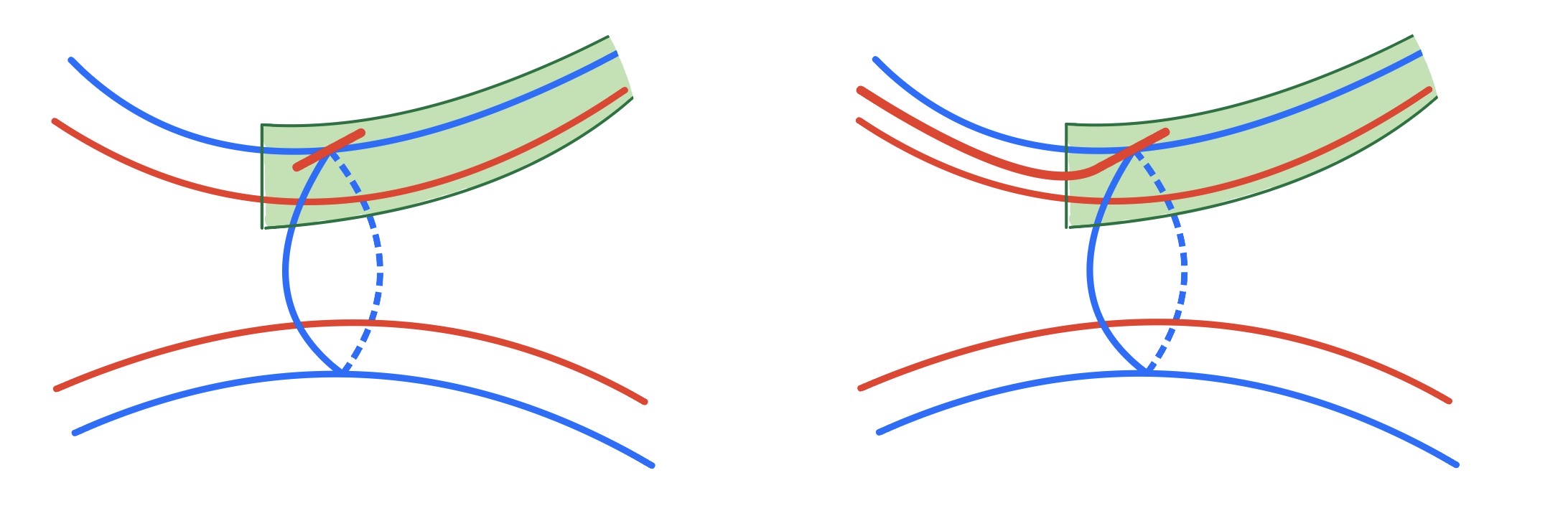}
    \caption{Case 1: the configuration of $B^u$ forces it to intersect $J$.}
    \label{fig: swath contradiction case 1}
\end{figure}

\begin{prop}
    If $B^u$ intersects $J$ at a point $x'$ so that $x>x^u_i$, either $x'$ is a fixed point of $F^u$, or we can find a closed transversal to $\Lam^u$. In either case, we obtain a contradiction.
\end{prop}
\begin{proof}
    If $x'$ is a fixed point of $F^u$, then $x^u$ is no longer the maximal fixed point. This is a contradiction to our assumption.

    Otherwise, if $x'$ is not a fixed point, we have $F(x')=x''$. Consider a path made up of four segments in the same unstable leaf: the loop $c_1$ representing the closed orbit, a path from a point $p\in c_1$ to $x'$ described by $F^u$, the path from $x'$ to $x''$, and a path from $x''$ to $p$. A pushoff of this closed curve, as pictured in Figure \ref{fig: transverse curve}, is a closed nullhomotopic transversal to the unstable lamination, which contradicts $\Lam^u$ being an essential lamination.
\end{proof}
\begin{figure}
    \centering
    \includegraphics[width=0.3\linewidth]{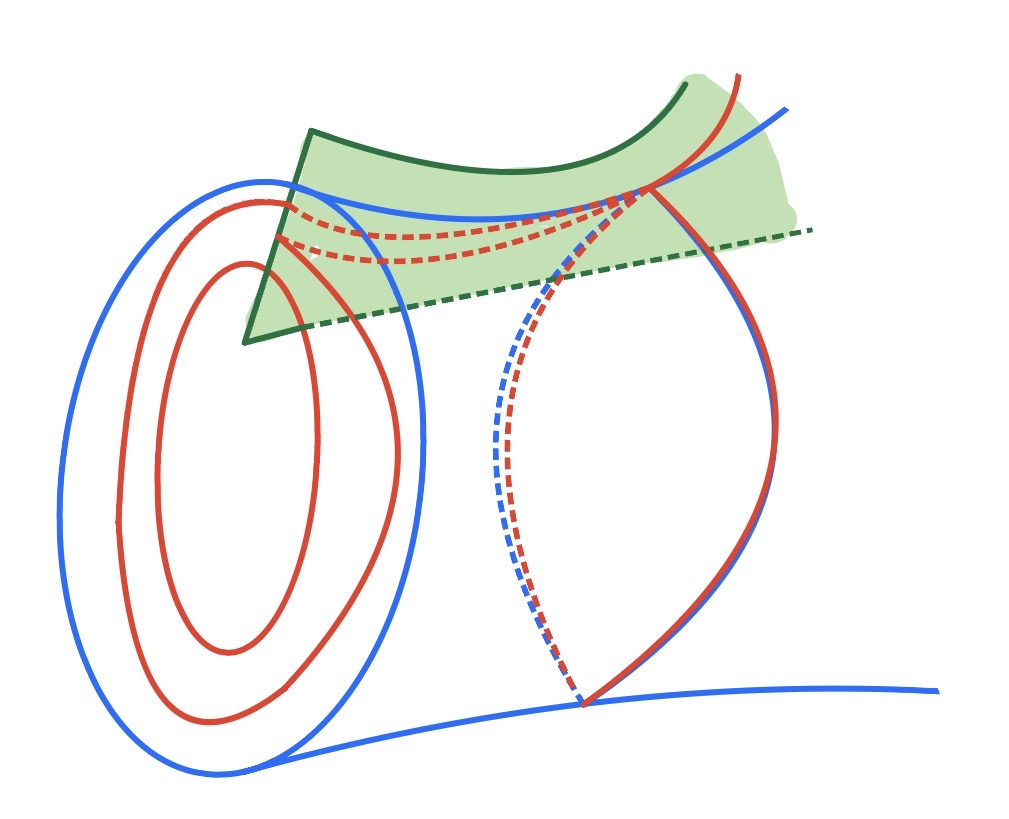}
    \caption{A closed curve that can be pushed off to a closed nullhomotopic transversal to an essential lamination $\Lam^u$.}
    \label{fig: transverse curve}
\end{figure}

\begin{figure}
    \centering
    \includegraphics[width=0.6\linewidth]{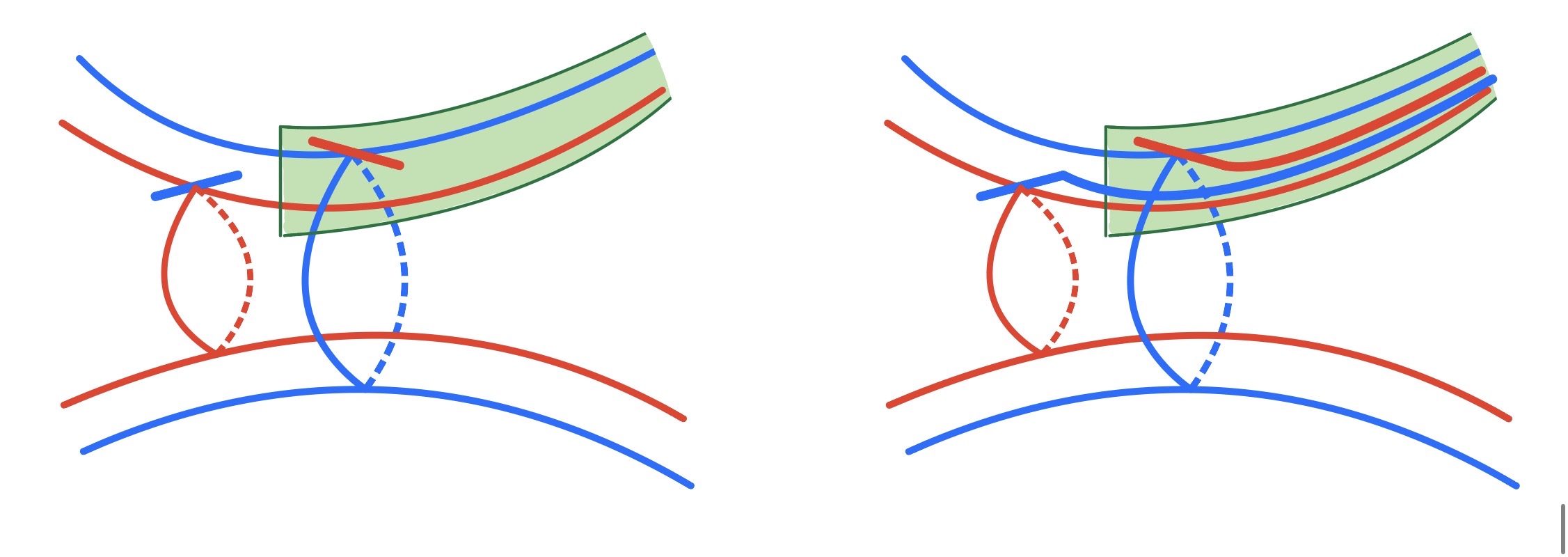}
    \caption{Case 2: both dual leaves are trapped in the fence.}
    \label{fig: swath case 2}
\end{figure}

\subsubsection*{Case 2} Suppose that $B^u$ crosses $A^s$ in the orientation as pictured in Figure \ref{fig: swath case 2}. Again, $B^u$ cannot cross $A^u$, and $B^s$ cannot cross $A^s$, so both dual leaves are trapped in the green fence. However, from Proposition \ref{prop: common fence}, the fence has width tending to $0$, which means that unstable leaves in the fence limit onto $A^u$ traveling to the end of $A^u$. However, in the end, $B^s$ lies between $A^u, B^u$, and $B^u$ would need to cross $B^s$ to limit onto $A^u$, which is impossible.

\subsubsection*{Case 3} Suppose the $B^u$ crosses $A^s$ as pictured in Figure \ref{fig: swath contradiction case 3}. In particular, the two transverse fences $J_1, J_2$ to the leaves $A^s, A^u$ satisfies that $J_1$ meets an end of $B^s$ and $J_2$ meets an end of $B^u$. Recall that $\gamma^s\subset A^s, \gamma^u\subset A^u$ are anti-homotopic closed orbits. By works of Fenley in the Anosov setting \cite[Theorem 3.3]{Fenley1995} and pseudo-Anosov setting \cite[Theorem 4.8]{Fenley1999}, lifts of a pair of anti-homotopic orbits to the universal cover are connected by a finite chain of lozenges. 
\begin{figure}
    \centering
    \includegraphics[width=0.5\linewidth]{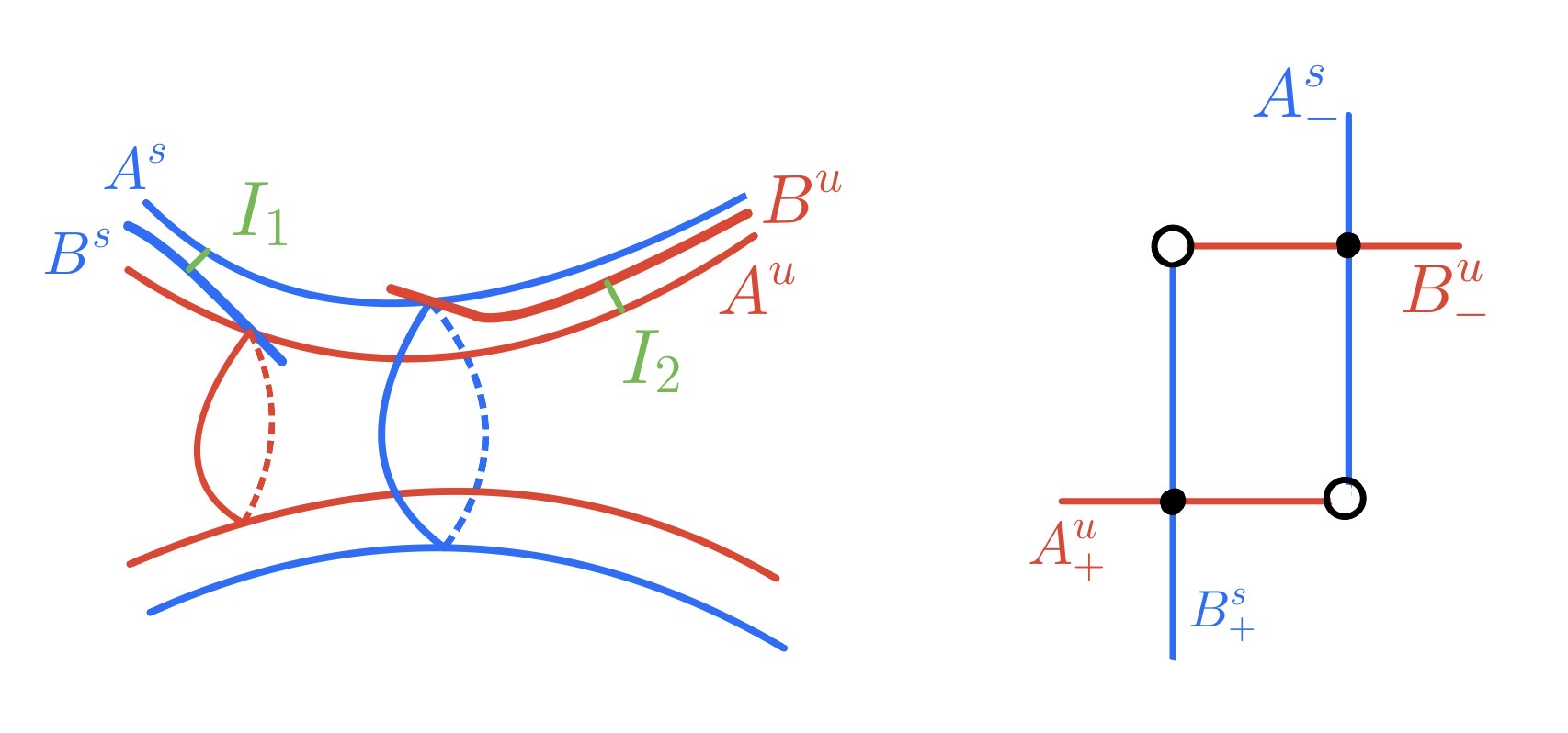}
    \caption{Case 3. Left: two transversals $J_1, J_2$ in green connecting $A^u, B^u$ that intersects neither $A^s, B^s$ nor the homotopy annulus $A^H$. Right: A standard picture of a lozenge for when the chain of lozenges is only one lozenge, which shows that the left picture is a contradiction.}
    \label{fig: swath contradiction case 3}
\end{figure}

We discuss two cases, one in which the chain of lozenges satisfies the property that consecutive lozenges only share a vertex, and the other in which there exist a sub-chain that is either $s$-scalloped or $u$-scalloped. In each case, we arrive at a contradiction.

\begin{lemma}
    Let $\mathcal{L}_1, \dots, \mathcal{L}_n$ be a chain of lozenges joining the periodic orbit $\gamma^s = A^s\cap B^u$ to $\gamma^u = B^s\cap A^u$. And suppose $\mathcal{L}_i$ meets $\mathcal{L}_{i+1}$ only at a corner. Such arrangement of lozenges is not compatible with the topology described in Case 3.
\end{lemma}
\begin{proof}
    Consider the 4 half-leaves represented by the four ends of the chain of lozenges, one end of each of $A^s, A^u, B^s, B^u$. In one lozenge, there is a cyclic ordering of the ends of the four leaves. When a chain of lozenges only intersect consecutive ones at corners, the cyclic ordering is preserved. We consider the case with one lozenge, and the argument extends to arbitrarily long chains.

    Consider also that the common transverse fence $J_1, J_2$ to either end of $A^s, A^u$ implies that there is a transversal $I_1$ to $\Lam^s$ identifying the end of $A^s$ intersecting $J_1$ with an end of $B^s$ also intersecting $J_1$. Similarly, there is a transversal $I_2$ identifying an end of $A^u$ to an end of $B^u$. In Figure \ref{fig: swath contradiction case 3}, $I_1$ is transverse to stable leaves to the left end, and $I_2$ is transverse to unstable leaves to the right end, so that $I_i$ a subset of a fiber of $J_i$. Choose labels for these ends, for example as in Figure \ref{fig: swath contradiction case 3}. Note that up to taking a double cover, the transversal $J_i$ preserves co-orientation, which implies that in the lozenge, the left end of $A^u$ should have the same sign as the left end of $B^u$. The same is true for $A^s, B^s$.

    Begin by examining the left end of $B^s$. Suppose it's the positive end $B^s_+$, then the left end of $A^s$ is also positive because $J_1$ identifies $A^s_+$ with $B^s_+$. The right end of $B^u$ is counterclockwise from $A^s_+$, and the right end of $A^u$ is clockwise from $B^s_+$. The right ends of $B^u$ and $A^u$ are identified by $J_2$. We see that such requirement is impossible to reconcile for any choice of labels for the right ends of $A^u, B^u$, and there are no valid labels for the ends compatible with the topology described by Case 3 if the lozenge structure is a chain with consecutive lozenges intersecting only at a corner.c
\end{proof}

\begin{figure}
    \centering
    \includegraphics[width=0.3\linewidth]{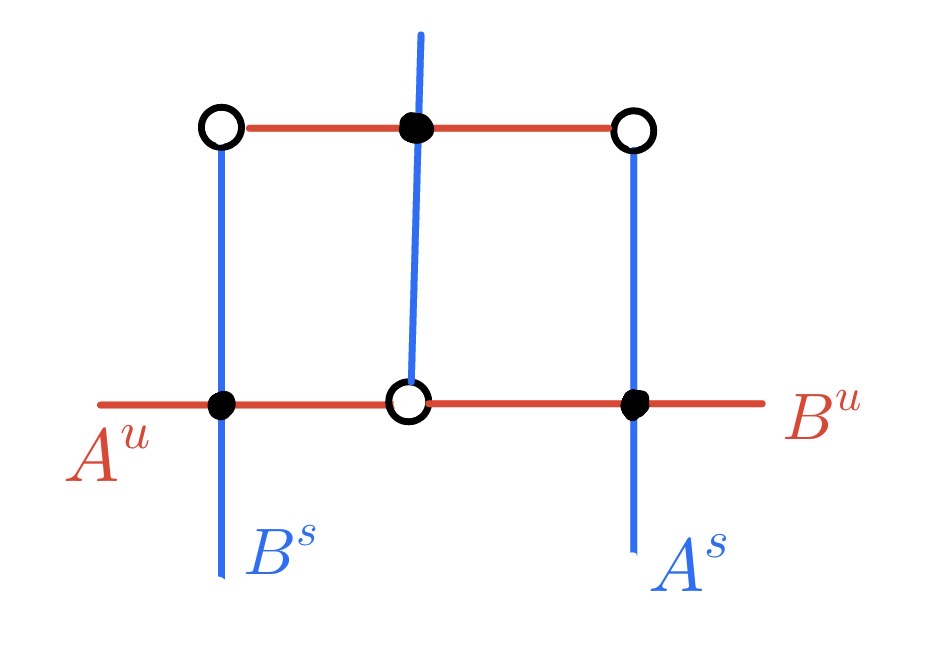}
    \caption{Two consecutive lozenges sharing a stable leaf.}
    \label{fig: adjacent lozenges}
\end{figure}

\begin{lemma}
    Let $\mathcal{L}_1, \dots, \mathcal{L}_n$ be a chain of lozenges joining the periodic orbit $\gamma^s = A^s\cap B^u$ to $\gamma^u = B^s\cap A^u$. And suppose there is some $i$ for which $\mathcal{L}_i$ meets $\mathcal{L}_{i+1}$ at a stable leaf. Such arrangement of lozenges is not compatible with the topology described in Case 3.
\end{lemma}
\begin{proof}
    There is a more fundamental obstruction in this case. In the two lozenges that share a stable leaf, the two unstable ends are non-separated, as shown in Figure \ref{fig: adjacent lozenges}. There exists no transverse path from one to the other preserving co-orientation. However, $J_2$ gives us such a path, which is a contradiction. The proof is analogous when lozenges share an unstable leaf.
\end{proof}

We have now proven that each of the cases for how $B^u, B^s$ could intersect $A^s, A^u$ has resulted in a contradiction with the topology imposed by the flow. We now conclude Theorem \ref{noswaths}, as desired.

\section{Corridor Reduction to Local Lemma}\label{section: corridor to local}

In this section, we prove that there is a sequence of isotopies of the stable and unstable Kneser branched surfaces after a uniformly bounded amount of splitting such that the parallel locus of any two leaves contain no radius $R$ is a union of connected components each covered by a radius $R$ disk.

The strategy is as follows: find a finite collection of smooth embedded arcs in the stable Kneser branched surface $\Sigma^s$ that partition every parallel locus between two leaves into components of radius less than $R$. Then, in a neighborhood of each arc, isotope the stable and unstable laminations to be transverse. This effectively deletes these neighborhoods from the parallel locus.

First, we show that doing a uniformly bounded amount of splitting will promote having no radius $R$ disks in the leaf-wise $R$-parallel locus to having no radius $R$ disks in the $R$-shadow.


\begin{prop}\label{prop: no radius R disks in shadow}
    Let $\Sigma^s, \Sigma^u$ be stable and unstable Kneser branched surfaces carrying the stable and unstable laminations $\Lam^s, \Lam^u$ of a pA flow $\phi$.  After a distance $2R$ splitting of $\Sigma^u$ to some new branched surface $\Sigma^u_1$, the $R$-shadow $P_R(\widetilde{L}^s, \cSigu_1;\cSigs)$ of $\cSigu_1$ on any leaf $\widetilde{L}^s$ does not contain any radius-$R$ disks. In other words, after splitting, the complement of $P_R(\widetilde{L}^s, \cSigu_1; \cSigs)$ in any leaf $\widetilde{L}^s$ is $R$-dense.
\end{prop}

\begin{proof}
    By Theorem \ref{noswaths}, the parallel locus $P(\widetilde{L}^s, \widetilde{L}^u;\cSigs, \cSigu)$ of any two leaves $\widetilde{L}^s, \widetilde{L}^u\subset \cLams, \cLamu$ with respect to $\cSigu$ contain no radius $R$ disks. We prove the contrapositive, that if the $R$-shadow $P_R(\widetilde{L}^s, \cSigu_1;\cSigs)$ contains a radius-$R$ disk $D^s$, then the $R$-parallel locus $P_R(\widetilde{L}^s, \widetilde{L}^u;\cSigs, \cSigu)$ is nonempty for some leaf $\widetilde{L}^u\subset \cSigu$.

    Taken an open cover of the radius-$R$ disk in the shadow and refine it to a finite cover $D^s\subset\bigcup_i\{D_i^s\}$, where $D^s_i\subset P_{R(i)}(\widetilde{L}^s, \widetilde{L}^u_i; \cSigs, \cSigu_1)$ is of radius $R(i)$. By definition of the parallel locus, we have a collection of disks $D^u_i\subset \widetilde{L}^u_i$ carried by the split branched surface $\cSigu_1$ so that $D_i^u$ is dual to $D_i^s$ for all $i$.

    Suppose open sets $X^s_1, X_2^s\subset \bigcup_i\{D_i^s\}\subset\widetilde{L}^s$ are dual to $X^u_1, X^u_2\subset \widetilde{L}^u_1, \widetilde{L}^u_2$, respectively, such that $X_1^s\cap X_2^s\neq\emptyset$, we show that $X^s_1$ has a continuation to a section of $\sat_s(X_2^s)\cup \sat_s(X_2^s)$ that is dual to some patch $X^u\subset \widetilde{L}^u_1$ that is a continuation of $X^1$.
    
    Consider the open set $Y^s = X_1^s\cap X_2^s$. It has two dual disks $Y_1^u\subset X^u_1, Y^u_2\subset X^u_2$ so that $Y^s$ is a section of both $\sat_u(Y_1^u)$ and $\sat_u(Y_2^u)$. This implies that $Y_1^u, Y_2^u$ are carried by the same unstable sector of $\cSigu_1$. Since any unstable disk has an unambiguous continuation to radius $2R$ due to the splitting, we can continue $X_1^u$ parallel to $X^u_2$ in the same unstable sector to a disk $Y^u$. Since the continuation is constructed to be carried by the same unstable sectors, the unstable saturations $\sat_u(X_1^u)\cup \sat_u(X^u_2) = \sat_u(Y^u)$. We conclude that $X_1^s\cup X_2^s$ is a section of $\sat_u(Y^u)$ because the continuation of $Y^s$ is dual to $X^u_2$.

    We have thus shown that a radius-$R$ disk in the shadow of the unstable branched surface split to distance $R$ implies that there is a leaf-wise parallel locus realizing that shadow. Since there are no leaf-wise parallel loci with radius-$R$ disks, there are also no shadows containing radius-$R$ disks.
\end{proof}

\begin{definition}
    A \textit{cutting arc} is a smooth embedded arc in the split stable branched surface $\Sigma^s_1$ with endpoints on the interior of sectors. A cutting arc is \textit{effective} for a leaf $\widetilde{L}^s$ if its endpoints are on interiors of sectors disjoint from the shadow of $\Sigma^u$ on the leaf $L^s$.
\end{definition}

\begin{lemma}\label{lemma: every leaf has cutting arc system}
    For each leaf $\widetilde{L}^s$, there is a locally finite graph $G_{L^s}$ of effective cutting arcs in $\widetilde{L}^s$ each of length at most $2R$ so that its complement in the shadow $P(\widetilde{L}^s, \cSigu_1)\setminus G_{L^s}$ consists of topological disks of diameter at most $2R$. The graph can be chosen to be equivariant with respect to deck transformations.
\end{lemma}
\begin{proof}
    This follows from Proposition \ref{prop: no radius R disks in shadow}, where we conclude that the shadow $P(\widetilde{L}^s, \cSigu_1)$ contains no disks of radius $R$. The clean locus is therefore $R$-dense. Choose a maximal set of clean points for which pairwise distances are at least $R$, and choose a locally finite cellulation of $\widetilde
    \Sigma^s$. The edges of the cellulation are by construction effective cutting arcs. The resulting components of $P(\widetilde{L}^s, \cSigu_1)\setminus G_{L^s}$ are disks of uniformly bounded diamter. Since the shadow is invariant under deck transformations that lie in the stabilizer of the stable leaf, we can choose cutting arcs to be equivariant. See Figure \ref{fig: cutting arcs} for an illustration.
\end{proof}

The next step is to find a uniformly finite family $\mathcal{C}$ of cuts on the stable branched surface $\Sigma^s$ so that for each lifted leaf $\widetilde{L}^s$, there is a subset $\mathcal{C}_{L^s}\subset \mathcal{C}$ of cuts effective in cutting the shadow $P(\widetilde{L}^s, \cSigu)$ into disks.

\begin{figure}
    \centering
    \includegraphics[width=0.5\linewidth]{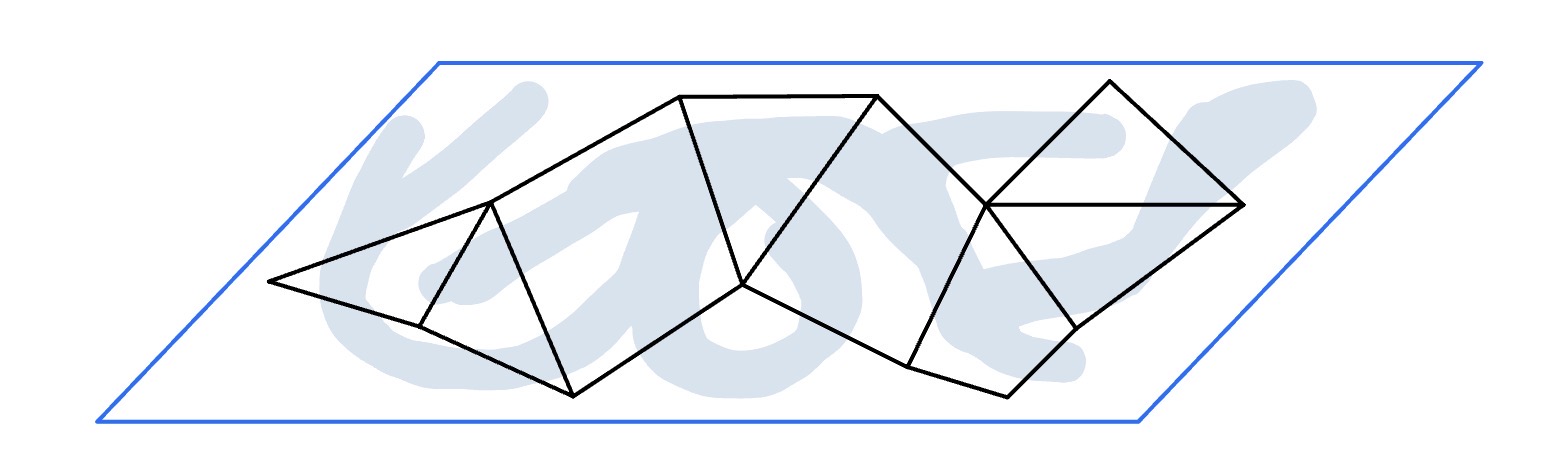}
    \caption{An illustration of black cutting arcs on a fundamental domain of a stable leaf $\widetilde{L}^u$, where the blue shaded areas represent the shadow $P(\widetilde{L}^s, \cSigu_1)$. The cutting arcs decompose the shadow into disks of bounded radius.}
    \label{fig: cutting arcs}
\end{figure}

\begin{definition}
    A point in a stable leaf $x\in L^s$ is \textit{clean} with respect to an unstable branched surface $\Sigma^u$ if it is disjoint from the parallel locus $P(\widetilde{L}^s, \cSigu)$.
    A point the stable branched surface $x\in\Sigma^s$ is \textit{clean} if in the universal cover, for every lift $\widetilde{x}$, every point in the stable fiber $p\in I_{\widetilde{x}}^s$ is disjoint from the parallel locus $P_R(\widetilde{L}^s(p), \cSigu;\cSigs)$, where $L^s(p)\subset\Lam^s$ is the stable leaf containing $p$.
\end{definition}


\begin{prop}\label{prop: uniformly finite clean enpoints}
    Let $\Sigma^s, \Sigma^u_1$ be fixed essential stable and unstable branched surfaces carrying the stable and unstable laminations $\Lambda^s, \Lambda^u$ normal to the fixed triangulation $\tau$. Assume that for every stable leaf $L^s\subset\Lambda^s$, the shadow $P(\widetilde{L}^s, \cSigu_1)$ contains no leaf-wise radius-$R$ disks. Then for every stable sector $S\subset \Sigma^s$, there exists a finite set 
    \[V_S = \{v_1, v_2,\dots, v_{N_S}\}\subset S\]
    depending only on $\Sigma^s, \Sigma^u_1, \tau, R, S$, so that the following property holds: For a stable leaf $L^s\subset\Lambda^s$ carried by $\Sigma^s$, if the system of cutting arcs given by Lemma \ref{lemma: every leaf has cutting arc system} has a clean endpoint in $S$, then there exists $v_i\in V_S$ such that the corresponding point in the leaf $x_i=L^s\cap I_{v_i}^s$ is clean in $L^s$. In other words, $x_i\in L^s\setminus P(\widetilde{L}^s, \cSigu_1)$. In particular, $V_S$ does not depend on the lamination $\Lam^s$ carried by $\Sigma^s$.
\end{prop}

\begin{proof}
    Consider a sector $S\subset \Sigma^s$. By construction, $\Sigma^u_1$ is a uniform splitting of a fixed Kneser branched surface, so it has bounded combinatorial complexity. In particular, there is some constant $N$ depending only on $R$, the branched surfaces $\Sigma^s, \Sigma^u_1$, and the triangulation $\tau$ so that at most $N$ sheets of the unstable lamination intersecting the fibered neighborhood of the stable sector $S$.

    In other words, we have a uniformly finite number of connected components $C_i$ of $N(S)\setminus \text{int} N(\Sigma^u_1)$. These are regions in the stable fibered neighborhood that lie between unstable fibered neighborhoods. Every $C_i$ is laminated by sheets of the stable lamination, and the combinatorial type of these sheets are finite. By combinatorial type, we mean the stable normal disk types in the triangulation together with the how it enters and exists $C_i$. Again, this data depends only on the two branched surfaces and the triangulation. For each passage combinatorial type $T$, choose finitely many points $V_{T}\subset S$ whose stable fibers meet every realization of $T$ as represented in the stable sheets. In general, a uniform number suffices because we have standard normal coordinates in a tetrahedron, all realizations of a fixed combinatorial passage type form an ordered family of parallel stable normal disks between a fixed collection of unstable normal walls. 

    Set $V_S = \bigcup_{C_i}\bigcup_T V_T$. Observe that now the cardinality of $V_S$ depends only on the  branched surfaces $\Sigma^s, \Sigma^u_1$, the constant $R$, and the triangulation $\tau$.

    Consider also that for each $v\in V_S$, the fiber $I_v^s$ decomposes into intervals of being clean or not only when it encounters the unstable sheets, so a consequence of there being a uniformly bounded number of unstable sheets interacting with $I_v^s$ is that the number of transitions is also uniformly bounded. In other words, now that we've established a uniformly finite set of candidate endpoints $V_S$ for cutting arcs, for each $v_i\in V_S$, the fiber $I_{v_i}^s$ is partitioned into uniformly finitely many connected components alternating between groups of stable plaques for which $v_i$ projects to a clean point and groups that see projections that are not clean.

\end{proof}

\begin{prop}\label{prop: unifromly finite mother graph of all cuts}
    There is a finite graph $G$ of smooth arcs on $\Sigma^s$ so that given a leaf $L^s$, there is a subgraph $H\subset G$ with a lift $H' = \pi_s^{-1}(H)$ to $L^s$ and another lift $\widetilde{H}'$ to the universal cover $\widetilde{L}^s$ satisfying the following properties:
    \begin{enumerate}
        \item Every arc $e\in E(H)$ is of length at most $2(R+C)$, where $C$ is the maximum diameter of the sectors in $\Sigma$;
        \item Endpoints of arcs $a_i\in E(H)$ are clean in $L^s$;
        \item Every connected component of $\widetilde{L}^s\setminus \widetilde{H'}$ is a disk of diameter at most $2(R+C)$
    \end{enumerate}
    
    Then, after a deck-equivariant isotopy, the cutting family is the full lift of a finite collection of arcs on $\Sigma^s$. The length condition and endpoint conditions are preserved in the universal cover.
\end{prop}

\begin{proof}
    For each section $S\subset \Sigma^s$, we use Proposition \ref{prop: uniformly finite clean enpoints} to obtain a set of potential clean endpoints $V_S$ for potential cutting arcs. Construct a graph $G$ on $\Sigma^s$ with vertices $V(G)=\bigcup_{S\subset \Sigma^s} V_S$, and for any pair of vertices $v, v'\in V(G)$, connect them with a smooth edge in $\Sigma^s$ if one of length at most $2(R+C)$ exists, where $C$ is the maximum diameter of all the sectors. 
    
    Let $L^s$ be a stable leaf in the stable lamination $\Lam^s$. Assuming that the scale of $R$ is much bigger than the diameter of the sectors (if not, replace $R$ with $100C$), we can modify the cellulation of the leaf by cutting arcs in Lemma \ref{lemma: every leaf has cutting arc system} such that each cutting arc has endpoints in different sectors. This increases the lengths of cutting curves by at most $2C$ and the area of each cell by at most $C^2\pi$. Then, replace vertices of cutting arcs with elements of $V_S$. For any original cutting arc in a leaf $L^s$ from Lemma \ref{lemma: every leaf has cutting arc system} that had endpoints on different sectors, the new endpoints are each at most distance $C$ from the original ones, and the smooth curve connecting them is of length at most $R+2C$. In general, the edges $E(G)$ may intersect each other on $\Sigma^s$. 

    As for connected components of $\widetilde{L}^s\setminus \widetilde{H'}$, consider the complementary regions of the original cutting arcs in Lemma \ref{lemma: every leaf has cutting arc system}. Vertices of these regions have moved at most distance $C$ each, which increases the diameter by at most $2C$. The region remains a disk.
\end{proof}

Next, we describe how to perform small blowups near these junctures of transition and thereby construct a new system of cutting arcs on the new split stable branched surface. 

\begin{lemma}
    For any point $v\in V_S\subset \Sigma^s$, the stable fiber $I_v^s$ can be written as a finite union of subintervals alternating between ones on which $v$ is clean in the leaf and ones where $v$ is in the shadow. Let $T_v$ be the finite set transition points.\qed
\end{lemma}

    Consider an edge $e\in E(G)$ with endpoints $v, v'$. Realizations of $e$ in stable leaves are parametrized by a Cantor set $K\subset I$, and a realization is admissible as a cutting arc if both endpoints are clean. A realization of $e$ on a stable leaf $L^s$ joins a point $x\in I_v^s$ to a point $y\in I_{v'}^s$, so the only places where the status of admissibility changes is at transition points $T_v, T_{v'}$. In the proof of Proposition \ref{prop: uniformly finite clean enpoints}, we show that there are uniformly finitely many such transition points. Blow up a small neighborhood in the stable direction near transition points to create complementary regions that are thickenings of disks. See Figure \ref{fig: local blowup to make disk complementary regions} for an illustration of a cross section of the resulting neighborhood of a new branched surface created after local blowups at transition points.
    
\begin{lemma}\label{lemma: create bubble and get finitely many branched surfaces}
    There are finitely many branched surfaces $\text{Bl}_i(\Sigma^s)$ obtained from local blowups each satisfying the following: 
    \begin{enumerate}
        \item For every stable lamination $\Lam^s_i$ carried by $\Sigma^s$, $\Lam_i^s$ is carried by some $\text{Bl}_i(\Sigma^s)$.
        \item The graph $G$ on $\Sigma^s$ descends to a graph $G'$ on the new branched surface $\text{Bl}_i(\Sigma^s)$ such that every vertex $v\in V(G')$ is clean with respect to the branched surface $\text{Bl}_i(\Sigma^s)$.
        \item Each edge $E(G')$ lifts to effective cutting arcs $e$ for leaves carried by $\text{Bl}_i(\Sigma^s)$.
    \end{enumerate}
    
\end{lemma}
\begin{proof}
    There are finitely many branched surfaces $\text{Bl}(\Sigma^s)$ resulting from local blowups at all possible transition points of each fiber $I_v^s$ for $v\in V(G)$. To see this, at each edge $e$ with a pair of endpoints $v, v'$, there are only finitely many orderings for $T_v, T_{v'}$ and therefore finitely many outcomes for the resulting branched surface up to isotopy near that edge. Since we only modify the branched surface according to finitely many edges, the resulting number of possible branched surfaces is also finite. See Figure \ref{fig: local blowup to make disk complementary regions} for one possible ordering of transition points.

    Every vertex in $G'$ being clean in $\text{Bl}(\Sigma^s)$ is by construction: when blowing up near a transition point, only assign a vertex to the side of the blowup for which the sub-fiber clean. Because we have blown up at every transition point, we can ensure that each vertex $V(G)$ descends to one point in $\text{Bl}(\Sigma^s)$. Connect two vertices $v, v'$ with a smooth arc if the arc can be made of length $R+2C$. The local splitting only changes curve lengths by an arbitrarily small amount. There are finitely many such edges, and the number is uniform independent of the stable laminations.
\end{proof}

\begin{figure}
    \centering
    \includegraphics[width=0.6\linewidth]{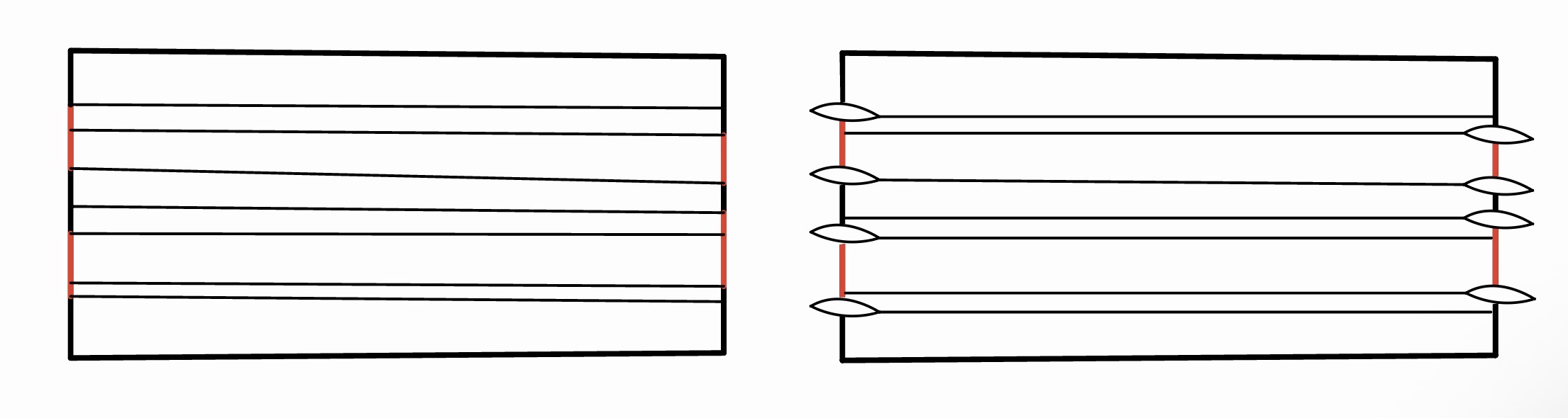}
    \caption{Left: an illustration of a possible ordering of transition points $T_v$ and $T_{v'}$. The left and right vertical edges are the fibers $I_v^s, I_{v'}^s$, with alternating black/red colors representing subintervals where the point is clean or in a shadow. There are four transition points on each side; Right: a cross section of the resulting fibered neighborhood of the new branched surface $\text{Bl}(\Sigma^s)$ after local blowups at transition points.}
    \label{fig: local blowup to make disk complementary regions}
\end{figure}

\subsection{Isotoping near cutting arcs}\label{subsection: isotoping near cutting arcs} In this subsection, we show that near a cutting arc, we can isotope the stable and unstable laminations while preserving their transversality to each other, so that we eliminate a neighborhood of the cutting arc from the shadow.

Consider an edge $e\in E(G')\subset \text{Bl}(\Sigma^s)$. After doing a distance $R+2C$ splitting on $\text{Bl}(\Sigma^s)$ to obtain $\Sigma^s_1$, we can arrange that we see no stable branching in a neighborhood of $e$. Consider the intersection of the unstable lamination with a neighborhood of $e$ in $\Sigma^s_1$, as illustrated in Figure \ref{fig: stable and unstable interaction near a cutting arc}.

\begin{figure}
    \centering
    \includegraphics[width=0.8\linewidth]{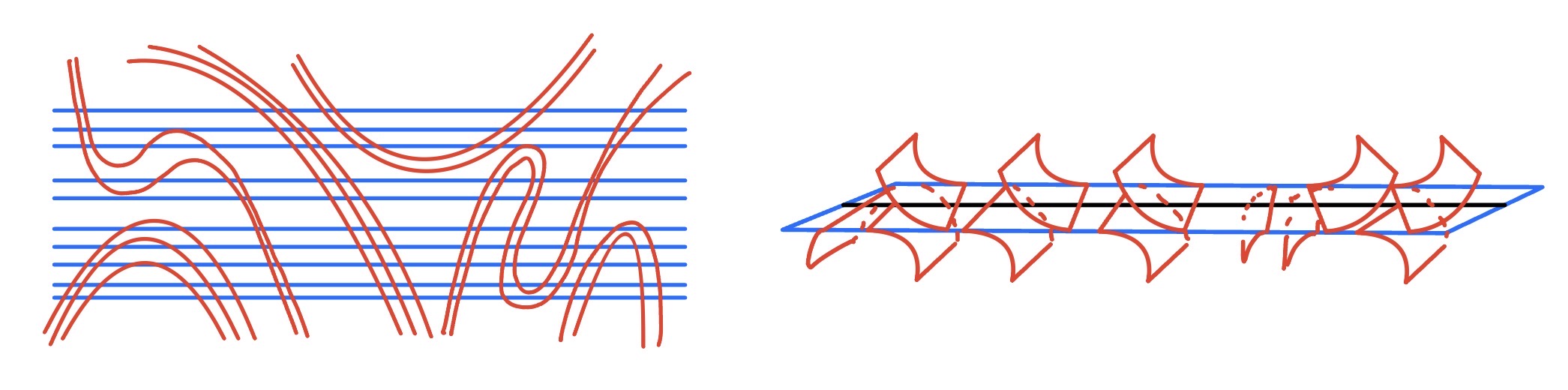}
    \caption{Left: a cross section of the stable leaves near a cutting arc $e$ along with interleaving unstable leaves; Right: a more global view of possible unstable branching in a neighborhood of $e$ as corresponding to the left picture.}
    \label{fig: stable and unstable interaction near a cutting arc}
\end{figure}

\begin{lemma}\label{lemma: straightening isotopy near a cut}
    Given an effective cutting arc $e\subset \Sigma_1^s$, there is an isotopy of $\Lambda^u$ supported in the fibered neighborhood $N(\Lam^s_i)$ such that afterwards, every point in a neighborhood of $e$ is clean.
\end{lemma}
\begin{proof}
    The stable lamination restricted to the stable saturation of a neighborhood of a cutting arc $\sat_s(U(e))$ consists of parallel thin disks parametrized by a Cantor set. The unstable lamination is transverse to the stable lamination. If $U(e)$ is thin, it is sufficient to consider a 2-dimensional cross section as illustrated in Figure \ref{fig: stable and unstable interaction near a cutting arc} (Left). 
    
    Consider a bigon $B$ bounded by an edge $e^s \subset L^s$ in a stable leaf and an edge $e^u\subset L^u$ in an unstable leaf. The two vertices $v, v'\in V (B)$ belong in the intersection $L^u\cap L^s$  and therefore lie in an orbit. We can find a curve $c$ in the orbit connecting $v$ to $v'$ as well as disks $D^s, D^u\subset L^s, L^u$ so that $\partial D^s = e^s\cup c, \partial D^u = e^u\cup c$. Together, there is a ball $B^3$ such that $\partial B^3 = D^s\cup D^u\cup B$. Isotope the unstable leaf $L^u$ guided by $B^3$ across $D^s$ to eliminate the bigon $B$. The isotopy preserves flow lines and transversality.

    Consider the maximal nested families of bigons as described by maximal clumps of stable an unstable leaves, similar to the process of simultaneously normalizing the stable and unstable laminations in the first place. The isotopies can be performed simultaneously by isotoping all the unstable sheets simultaneously. The isotopies are supported on a small neighborhood of the ball $B^3$. Different maximal clumps produce disjoint guiding balls.

    \begin{figure}
        \centering
        \includegraphics[width=0.6\linewidth]{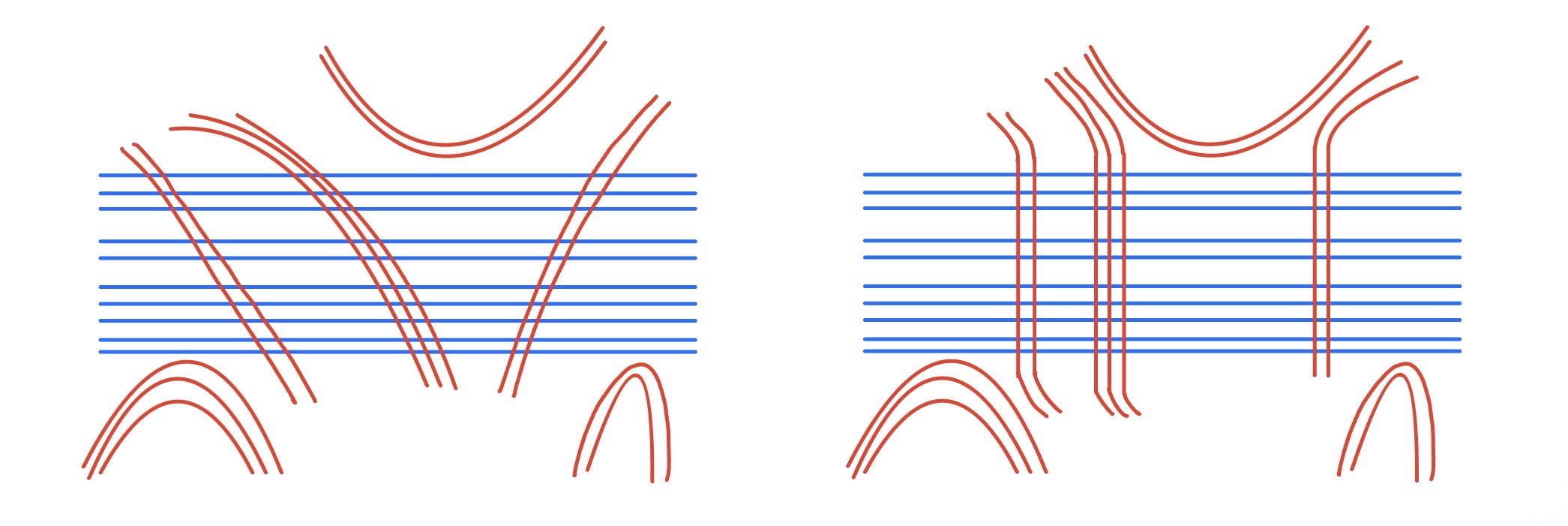}
        \caption{The result of aligning unstable sheets to coincide with the stable normal fibers.}
        \label{fig: cutting arc after bigon reduction}
    \end{figure}

    After all bigon reduction isotopies, each local sector of the unstable branched surface $\Sigma^u_1$ either becomes disjoint from the stable sheets or intersects the entire neighborhood transversely, see Figure \ref{fig: cutting arc after bigon reduction} (Left) for an illustration. In the transverse intersection case, perform a further isotopy to align the unstable sheets to coincide with the stable normal fibers, as shown in Figure \ref{fig: cutting arc after bigon reduction} (Right). This way, the neighborhood $\sat(U(e))$ after isotopy is now disjoint from the shadow $P(L^s, \cSigu_1)$ for every leaf $L^s\in\Lam^s$. 
    
    Moreover, the unstable isotopy does not change the shadow of leaves near any cuts that have already been isotoped outside the neighborhood $U(e)$. To see this, note that during the isotopy eliminating bigons, we arrange for $B^3$ to be contained in the fibered neighborhood of $\Sigma^s_1$, which is possible as a consequence of Corollary \ref{cor: no bigons in vertical boundary}, since $B^3$ being outside the fibered neighborhood implies a stable-unstable bigon in the vertical boundary of the fibered neighborhood of the stable branched surface $\partial_vN(\Sigma^s_1)$. Therefore, every bigon elimination isotopy is supported in the fibered neighborhood of $\Sigma^s_1$, and no bigons exists near any cuts that has already undergone isotopy.
\end{proof}

Upon performing the isotopy in Lemma $\ref{lemma: straightening isotopy near a cut}$ for every edge $e\in E(G')$, a neighborhood of the saturation of $G'$ is disjoint from the shadow $P(\widetilde{L}^s, \cSigu_1)$ for every $L^s\subset \Lambda^s$. Next, we consider a decomposition of the fibered neighborhood $N(\Sigma^s_1)$ along these clean neighborhoods into simple pieces on which we can study the stable and unstable interactions further. We prove the following decomposition lemma.

\begin{lemma}\label{lemma: decomposition into local lemma pieces}
     The fibered neighborhood $N(\Sigma^s_1)$ of the stable branched surface, obtained from the Kneser branched surface $\Sigma^s$ by splitting and blowups, can be written as a finite union of thickened disks $D^2_i\times I_i$ such that for each thickened disk, the unstable branched surface $\Sigma^u_1$ is transverse to the stable leaves at the boundaries $\partial D_i^2\times I$. Further, this decomposition only depends on  $N(\Sigma^s_1)$ and is therefore independent of the underlying stable lamination $\Lam^s$.
\end{lemma}

\begin{proof}
    Consider the saturation of the cutting arcs $V(G')$, and cut $N(\Sigma^s_1)$ along the union of all saturations $\sat_s(V(G'))$. Since the cutting graph $G'$ is finite, there are only finitely many cycles bounding disks in $\Sigma^s_1$. For each such disk region, there are at most finitely many parallel copies in the fiber direction because every stable fiber $I_v^s$ of $\Sigma^s_1$ has only finitely many transition points. Since a neighborhood $U(e)$ of any cutting arc $e\in E(G')$ is disjoint from the shadow, any intersection of the unstable branched surface with $\Sigma^s_1$ near $U(e)$ must be aligned with the stable normal fibers by the construction in Lemma \ref{lemma: straightening isotopy near a cut}.
\end{proof}

\section{The Local Lemma}\label{section: local lemma}
In this section, we describe how to show that after decomposing the fibered neighborhood $N(\Sigma^s_1)$ of the stable branched surface into thickened disks $D^2_i\times I_i$ as in Lemma \ref{lemma: decomposition into local lemma pieces}, we can ensure that the unstable lamination $\Lam^u$ interacts with each thickened disk piece as a trivial product. This leads us to construct finitely many pairs of transverse split Kneser branched surfaces, and Proposition \ref{prop: transverse Kneser pair determines flow} shows that each pair uniquely determines a pseudo-Anosov flow.

\begin{definition}[Dynamic splitting]
    Let $\Sigma^s, \Sigma^u$ be Kneser branched surfaces carrying the stable and unstable laminations $\Lam^s, \Lam^u$, respectively. A dynamic splitting of $\Sigma^s, \Sigma^u$ is a dynamic pair $(\Sigma^s_D, \Sigma^u_D)$ such that $\Sigma^s_D$ (resp, $\Sigma^u_D$) is a splitting of $\Sigma^s$ (resp, $\Sigma^u$).
\end{definition}

\begin{theorem}\label{thm: local lemma}
    Let $X \cong D^2\times I$ be a subset of $M^3$ such that the stable branched surface lies in $X$ as $D^2\times t$ sheets except in a small neighborhood of the branch locus, and let $\Sigma_2^u$ be transverse to $\Sigma^s_2$ on the cylindrical boundary $\partial D^2\times I$ at a uniformly bounded number of sheets. Then upon uniformly bounded finite splitting, we obtain pairs of transverse branched surfaces $\Sigma^u_2, \Sigma_2^s$ such that inside $X$, it's one of finitely many possible pairs. The number of possibilities depends only on the number of unstable sheets intersecting the cylindrical boundary. 
\end{theorem}

\begin{proof}
    Let $\Lam^s, \Lam^u$ be a pair of stable and unstable laminations carried by the stable and unstable branched surfaces. For each connected component of $\Lam^u\cap X$, call it a local leaf $l$, the product structure on $X$ defines a Morse function $f:l\rightarrow [0,1]$ on $l$. If there are no critical points for all $l$, then the unstable lamination $\Lam^s$ is a product inside $X$. 
    
    If there exists an index 0 or index 2 critical point, there exists a pair of stable and unstable sheets that intersect in a nullhomotopic closed curve $c$ inside $X$, as shown in Figure \ref{fig: local lemma index 1 2 3}. Since the two laminations are the stable and unstable laminations of a pseudo-Anosov or Anosov flow, $c$ must be an orbit of the corresponding flow. However. there are no nullhomotopic closed flow lines, so this is a contradiction to Anosov or pA dynamics. 
    
    If there is some $l$ with an index 1 critical point, then in a neighborhood of that critical point, there is a local saddle. Consider the intersection between $l$ with all the local stable leaves. By our assumption, there is a local stable leaf lying above the saddle level and below it, as pictured in Figure \ref{fig: local lemma index 1 2 3}. Intersections of stable and unstable leaves are segments of flow lines of some pA flow $\phi$. Choosing a direction of the flow to ensure that the flow lines are connected, we can find in the unstable local leaf two simple closed curves $\gamma_1, \gamma_2$, as shown in Figure \ref{fig:local lemma index 1 contradiction}, with one geometric intersection. However, this is a contradiction because unstable leaves are globally planes or annuli, and such curves cannot exist. 

    Therefore, the combinatorial matching of endpoints on $\partial D^2\times t$ of the unstable branched surface persists as $t$ varies and determines the matching for all levels. The number of matchings determine the number of possible unstable branched surfaces restricted to $X$.
\end{proof}

In practice, using Lemma \ref{lemma: decomposition into local lemma pieces}, we choose $X$ to be a thickened disk $D_i\times I$, being careful to restrict $X$ to a small neighborhood of the stable sheets in the $I$ direction, since saddles can occur atop the topmost stable leaf in a thickened disk region. Nevertheless, since the number of unstable endpoints on $\partial D_i$ is uniformly bounded, there are only finitely many ways to complete the unstable branched surface away from the stable neighborhoods.

The following corollary is a result of Lemma \ref{lemma: decomposition into local lemma pieces}, that there is a finite partition of the fibered neighborhood $N(\Sigma^s_1)$.

\begin{corollary}\label{cor: finite transverse kneser pairs}
    There are finitely many transverse pairs $(\Sigma_i^s, \Sigma_i^u)$ of essential branched surfaces carrying the stable and unstable laminations of any Anosov or pseudo-Anosov flow on $M$.
\end{corollary}
\begin{proof}
    By Lemma \ref{lemma: decomposition into local lemma pieces}, the fibered neighborhood $N(\Sigma^s)$ can be cut along the saturations of finitely many cutting arcs into finitely many product regions $D_j\times I_j$ drawn from a finite collection of possible types. For each product region, $\Sigma^u$ meets the vertical boundary in a uniformly bounded number of arcs.

There, in each $D_j\times I_j$, there are only finitely many possible combinatorial types for the intersection with $\Sigma^u$. Since there are only finitely many product regions and finitely many ways to identify their vertical boundary faces, there are only finitely many possibilities for the resulting transverse pair $(\Sigma^s,\Sigma^u)$.

The bound depends only on the finite collection of local pieces and their uniformly bounded boundary intersection data. It therefore depends only on $M$ and the fixed triangulation $\tau$, not on the carried stable and unstable laminations.
\end{proof}

To complete the proof of Theorem \ref{Theorem: main theorem}, we prove the following result that a transverse pair of stable and unstable branched surfaces determine the flow. 

\begin{prop}\label{prop: transverse Kneser pair determines flow}
    Let $(\Sigma^s,\Sigma^u)$ be a transverse pair of essential branched surfaces in $M$. Suppose two Anosov or pseudo-Anosov flows $\phi_1,\phi_2$ have stable and unstable laminations simultaneously carried by $\Sigma^s,\Sigma^u$. Then $\phi_1$ and $\phi_2$ have the same free homotopy classes of orbits and are thus orbit equivalent.
\end{prop}

Before proving Proposition \ref{prop: transverse Kneser pair determines flow}, we first establish that in an unstable leaf R,  intersections with the vertical boundary of the stable branched surface $\Sigma^s$ is $K$-dense in R.

\begin{lemma}\label{lemma: stable maw is dense in unstable leaf}
    Let $R$ be an unstable leaf transverse to the a normal Kneser branched surface $\Sigma^s$. Then there exists $K=K(\Sigma^s,\Sigma^u)$ such that the intersection between $R$ and the vertical boundary $R\cap\partial_v N(\Sigma^s)$ is $K$-dense in $R$.\qed
\end{lemma}


    \begin{figure}
        \centering
        \includegraphics[width=0.7\linewidth]{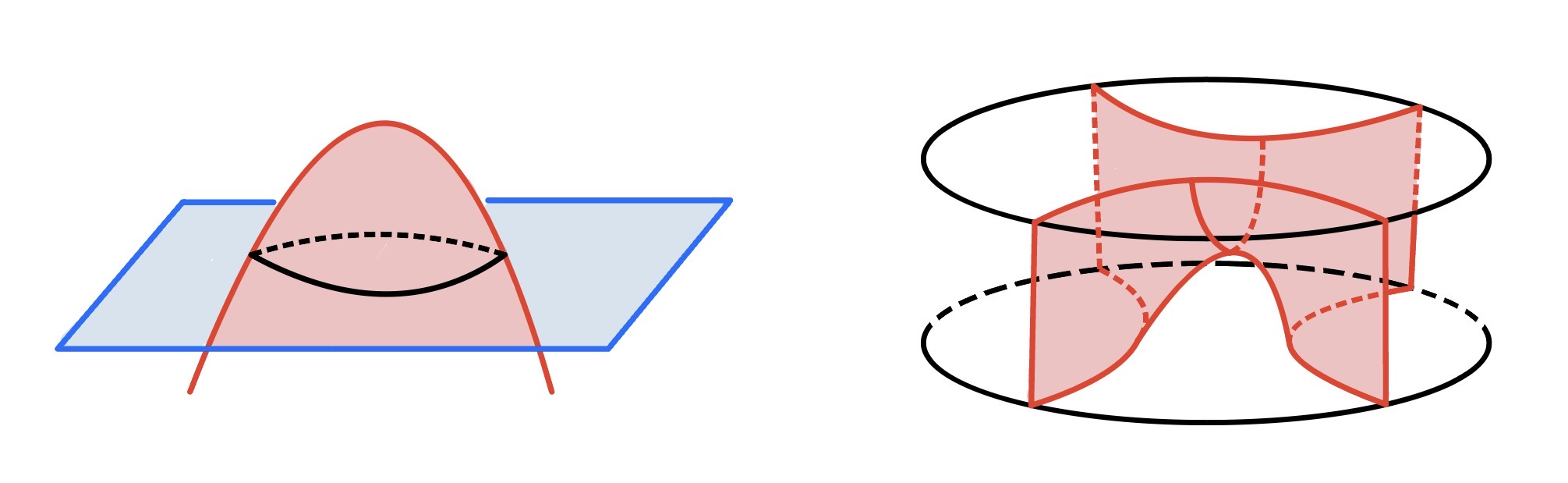}
        \caption{A local picture of index 0 and 2 critical points in a local unstable leaf and a local picture of an index 1 critical point.}
        \label{fig: local lemma index 1 2 3}
    \end{figure}

\begin{figure}
        \centering
        \includegraphics[width=0.8\linewidth]{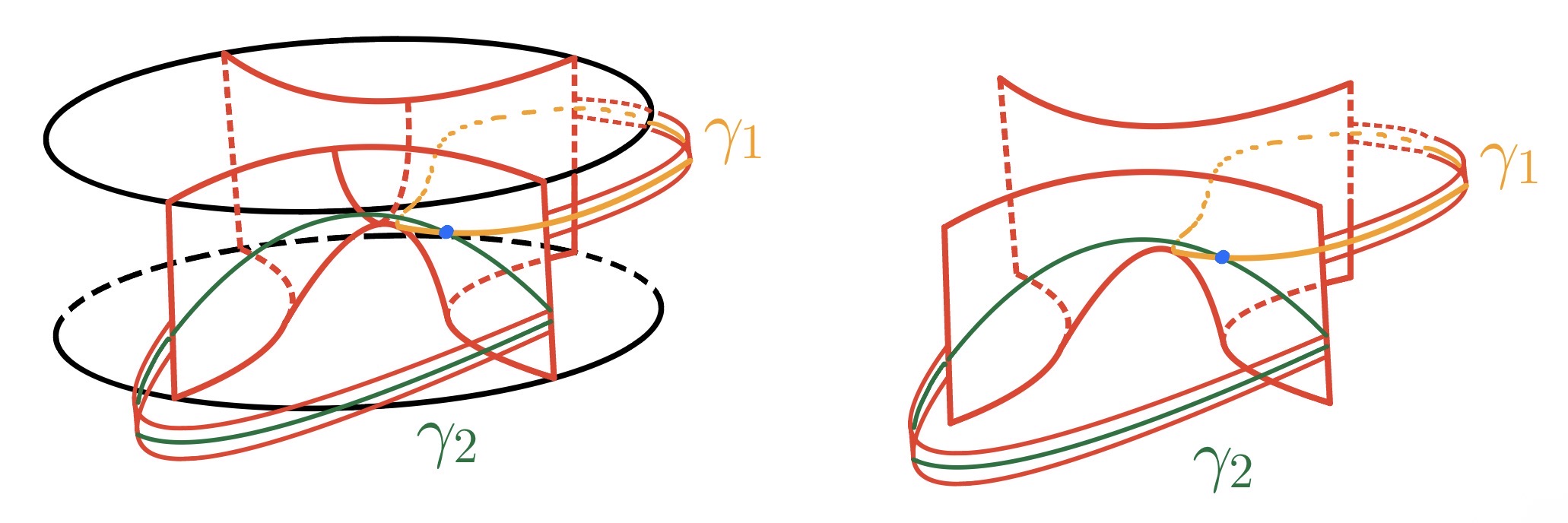}
        \caption{Two curves $\gamma_1, \gamma_2$ on an unstable leaf with one intersection.}
        \label{fig:local lemma index 1 contradiction}
\end{figure}

\begin{figure}
    \centering
    \includegraphics[width=0.5\linewidth]{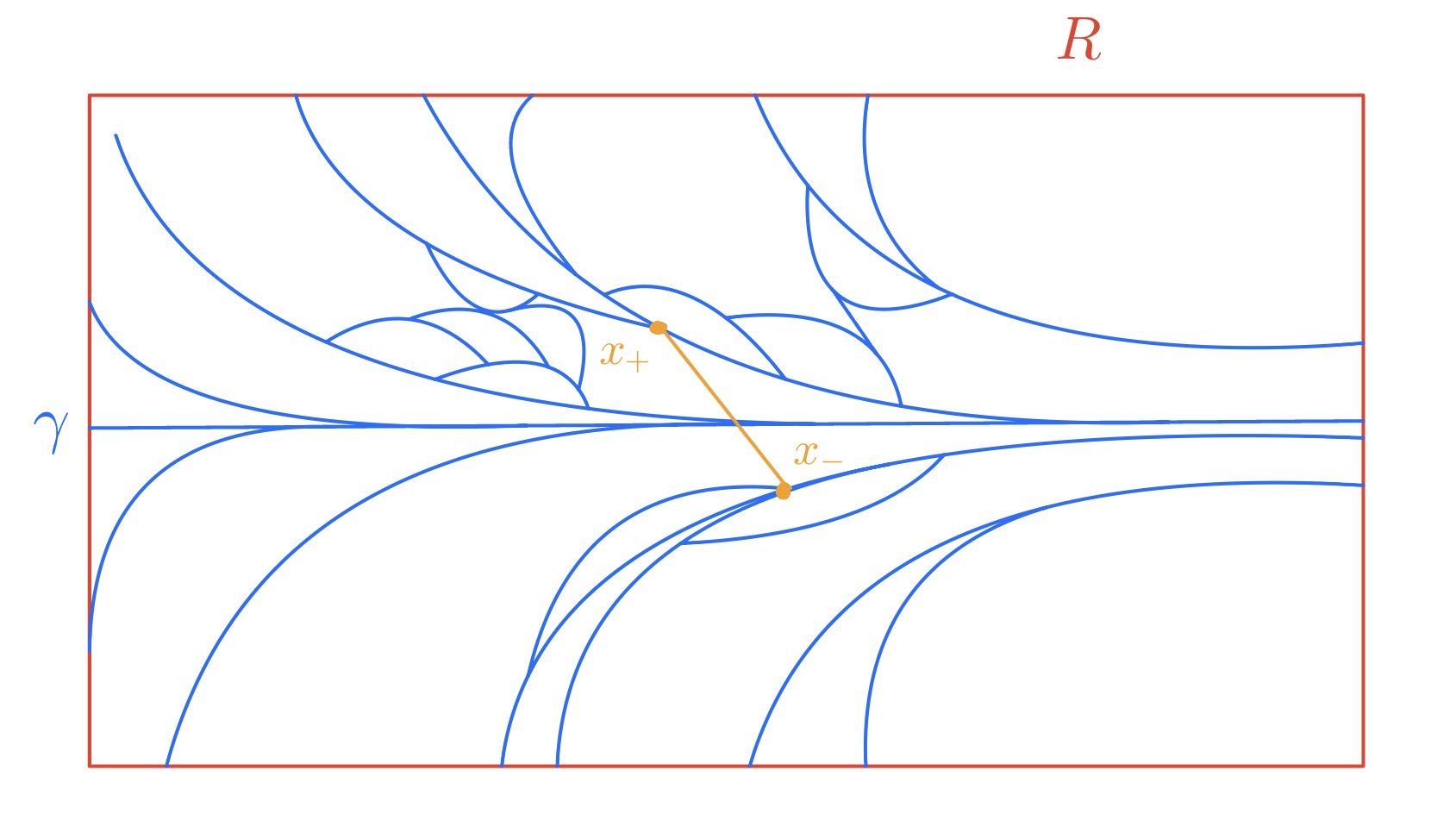}
    \caption{An unstable leaf $R$ of $\Lam^u_1$ and the train track $\tau$ as the imprint of $\Sigma^s$ on $R$. The orange transversal $J$ connects two intersections of $R$ with the non-bubble branch locus of $\Sigma^s$.}
    \label{fig: orbit equivalence}
\end{figure}
Now, we can prove Proposition \ref{prop: transverse Kneser pair determines flow}. 
\begin{proof}[Proof of Proposition \ref{prop: transverse Kneser pair determines flow}]
    Suppose that $\Sigma^s, \Sigma^u$ have dynamical splittings $\Sigma^s_1, \Sigma^u_1$ and $\Sigma^s_2, \Sigma^u_2$ carrying stable and unstable laminations $\Lam^s_1, \Lam^u_1$ and $\Lam^s_2, \Lam^u_2$ of pseudo-Anosov flows $\phi_1, \phi_2$, respectively. Let $\gamma$ be a closed orbit of $\phi_1$ such that $\gamma = B\cap R$, where $B\subset \Lam_1^s$ is a stable leaf and $R\subset \Lam^u_1$ is an unstable leaf. We show that $\Lam^s_2$ has a leaf $L$ containing a closed curve $c$ in the same free homotopy class as $\gamma$.

    Consider the train track $\tau = \Sigma^s\cap R$. Every non-periodic flow line in the unstable annulus $R$ is asymptotic to $\gamma$ in backward time and therefore spirals toward $\gamma$ intrinsically in $R$. The intersection between $R$ and the vertical boundary $R\cap \partial_vN(\Sigma^s)$ is $K$-dense for some $K>0$ by Lemma \ref{lemma: stable maw is dense in unstable leaf}. See Figure \ref{fig: orbit equivalence} for an illustration. Since $\gamma$ is separating on $R$, denote the two components of $R\setminus \gamma$ as $R_+, R_-$. On each of $R_+, R_-$, we choose an intersection in $R\cap \partial_vN(\Sigma^s)$ distance at most $K$ away from $\gamma$ that border a non-bubble complementary component of $\tau$. Recall that finitely many bubble complementary regions were created in Lemma \ref{lemma: create bubble and get finitely many branched surfaces}. On the other hand, a non-bubble complementary region is non-compact and has two boundary rays that are branches of $\tau$ realized as intersection of $R$ with the horizontal boundary $\partial N_h(\Sigma^s)$. Since the branches of $\tau$ carry flow lines spiraling to $\gamma$, the boundary rays of non-bubble complementary regions also spiral into $\gamma$. Let the chosen intersection in $R_\pm$ be $x_\pm$. Let $J$ be the transversal joining $x_\pm$. Since the boundary rays converge to $\gamma$, $J$ has the property that following the holonomy of $\Lam_1^s\cap R$ around the core of $R$ give a first-return map that brings both endpoints of $J$ closer to $\gamma$.

    Next, consider the intersection train track $\tau' = \Sigma_2^s\cap R$, which is carried by $\tau$ as a splitting. Following leaves of $\Lam_2^s\cap R$ once around $R$ also gives an order-preserving return map $f_2$ on $K=J\cap\Lam_2^s$, where spiraling boundary rays for $\tau$ force the return map inward at the two ends. Therefore the first-return map $f:K\rightarrow K$ has a fixed point, and the corresponding leaf of $\Lam^s_2$ has a closed curve $c\subset R$ with $[c]=[\gamma]$. That stable leaf of $\phi_2$ therefore is an annulus and has a closed orbit $\gamma'$ such that $[\gamma']=[\gamma]$.

    We can similarly argue by interchanging the roles of $\phi_1, \phi_2$ that any closed orbit in one flow has a corresponding closed orbit of the same free homotopy class in the other. Work of Barthelm\'e--Frankel--Mann \cite{BFMclassification} tells us that transitive flows on 3-manifolds without trees of scalloped regions are determined by the free homotopy classes of their closed orbits. Since $M$ is hyperbolic, every Anosov and pseudo-Anosov flow is transitive, and there are no trees of scalloped regions, so $\phi_1, \phi_2$ are orbit equivalent. We have thus shown that a transverse pair of split Kneser branched surfaces can carry the stable and unstable laminations of at most one flow up to orbit equivalence.
\end{proof}

Since Corollary \ref{cor: finite transverse kneser pairs} established that there are finitely many transverse pairs of essential branched surfaces carrying the stable and unstable laminations of Anosov or pseudo-Anosov flows, and each pair determines at most one flow, we conclude Theorem \ref{Theorem: main theorem}.

\subsection*{Tool use} AI was not used at any stage in the creation and preparation of this work.

\bibliography{references}
\bibliographystyle{plain}

\end{document}